\RequirePackage{snapshot}
\documentclass[11pt,final]{siamart0516}
\usepackage{microtype}
\usepackage[text={5.125in,8.25in},centering]{geometry}
\usepackage{cite}
\usepackage{amsmath, amsfonts, amssymb}
\usepackage{graphicx}
\usepackage{braket}
\usepackage{enumerate}
\usepackage{mathtools} \mathtoolsset{centercolon}
\usepackage[mathscr]{euscript}
\usepackage{array}
\usepackage{subdepth}
\usepackage{xcolor}
\usepackage{caption}
\usepackage{subcaption}
\usepackage{booktabs}
\usepackage{pgfplots} %
\usepackage{xfrac}

\usepgfplotslibrary{external}
\crefalias{subsection}{section}
\crefalias{subsubsection}{section}
\numberwithin{equation}{section}
\numberwithin{theorem}{section}
\numberwithin{table}{section}
\numberwithin{figure}{section}
\definecolor{amethyst}{rgb}{1, 0, 1}
\definecolor{blue-violet}{rgb}{0.54, 0.17, 0.89}
\definecolor{brightturquoise}{rgb}{0.03, 0.91, 0.87}
\definecolor{darkgreen}{rgb}{0.08,0.4,0.1}

\newcommand{\conj}{{\star}}

\DeclareMathOperator*{\argmin}{\mathrm{argmin}}
\DeclareMathOperator*{\argmax}{\mathrm{argmax}}
\DeclarePairedDelimiterX{\ip}[2]{\langle}{\rangle}{#1,#2}
\newcommand{\T}{^T\!}
\newcommand\textt[1]{\hbox{\quad#1\quad}}
\newtheorem{example}[theorem]{Example}

\newcommand{\xbar}{\overline{x}}
\newcommand{\ybar}{\overline{y}}
\newcommand{\R}{\mathbb{R}}

\newcommand{\Rbar}{\overline{\mathbb{R}}}
\newcommand{\Rn}{{\R}^n}
\DeclareMathOperator{\ri}{\mathrm{ri}}
\DeclareMathOperator{\dom}{\mathrm{dom}}
\DeclareMathOperator{\inter}{\mathrm{int}}
\DeclareMathOperator{\cl}{\mathrm{cl}}
\DeclareMathOperator{\val}{\mathrm{val}}
\DeclareMathOperator{\range}{\mathrm{range}}
\DeclareMathOperator{\cone}{\mathrm{cone}}
\DeclareMathOperator{\epi}{\mathrm{epi}}
\DeclareMathOperator{\ncone}{\Nscr}

\newcommand{\gam}{\gamma}

\newcommand{\conv}{{\mbox{conv}\,}}

\newcommand{\sign}{\mathrm{sign}}
\newcommand{\gauge}[1]{\gamma_{\mbox{\scriptsize$#1$}}}

\newcommand{\sd}{\partial}

\DeclareMathOperator{\prox}{\mathrm{prox}}
\DeclareMathOperator*{\minim}{\mathrm{minimize}}
\newcommand{\minimize}[1]{\displaystyle\minim_{#1}}
\DeclareMathOperator*{\maxim}{\mathrm{maximize}}
\newcommand{\maximize}[1]{\displaystyle\maxim_{#1}}
\DeclareMathOperator{\st}{\mathrm{subject\ to}}
\newcommand{\pp}{{\sharp}}

\newcommand{\Cscr}{\mathcal{C}}
\newcommand{\Fscr}{\mathcal{F}}
\newcommand{\Hscr}{\mathcal{H}}
\newcommand{\Nscr}{\mathcal{N}}
\newcommand{\Qscr}{\mathcal{Q}}
\newcommand{\Uscr}{\mathcal{U}}

\newcommand{\lagrange}{(L_p)}
\renewcommand{\L}{_{\scriptscriptstyle L}}

\newcommand{\half}{{\textstyle{\frac{1}{2}}}}

\newcommand\fullwidthdisplay{\displayindent0pt \displaywidth\columnwidth}
\AtBeginDocument{
  \everydisplay\expandafter{\expandafter\fullwidthdisplay\the\everydisplay}
}

\title{Foundations of gauge and perspective duality\footnote{J\lowercase{une 18, 2018}}}

\author{A.Y. Aravkin\thanks{Department of Applied Mathematics, University of
    Washington, Seattle (\url{sasha.aravkin@gmail.com}). Research supported by the
    Washington Research Foundation Data Science Professorship.} \and J.V.
  Burke\thanks{Seattle, WA
    (\url{jvburke01@gmail.com}). Research supported in part by NSF award
    DMS-1514559.} \and D. Drusvyatskiy\thanks{Department of Mathematics,
    University of Washington, Seattle (\url{ddrusv@uw.edu}; \url{kmacphee@uw.edu}). Research partially
    supported by AFOSR YIP award FA9550-15-1-0237.} \and M.P.
  Friedlander\thanks{ Departments of Computer Science and Mathematics,
    University of British Columbia, Vancouver, BC, Canada (\url{mpf@cs.ubc.ca}).
    Research supported by ONR award N00014-16-1-2242.} \and \mbox{K.J.
  MacPhee$^{\S}$}}
\begin{document}
\maketitle
\thispagestyle{plain}
\pagestyle{myheadings}

\begin{abstract}
  We revisit the foundations of gauge duality and demonstrate that it
  can be explained using a modern approach to duality based on a
  perturbation framework. We therefore put gauge duality and
  Fenchel-Rockafellar duality on equal footing, including explaining gauge dual
  variables as sensitivity measures, and showing how to recover primal
  solutions from those of the gauge dual. This vantage point allows a
  direct proof that optimal solutions of the Fenchel-Rockafellar dual
  of the gauge dual are precisely the primal solutions rescaled by the
  optimal value. We extend the gauge duality framework
  to the setting in which the functional components are general
  nonnegative convex functions, including problems with piecewise
  linear quadratic functions and constraints that arise from
  generalized linear models used in regression.
\end{abstract}



\section{Introduction}\label{sec:intro}

Sensitivity of the optimal values and solutions of optimization problems, 
with respect to perturbations in the problem data, is a central concern of 
Fenchel-Rockafellar duality theory. Lagrange duality can be regarded as a special
case of this theory, in which perturbations to the data are introduced
in a particular manner. Gauge duality, on the other hand, as introduced in 1987 by
Freund~\cite{freund}, was developed without any
reference to sensitivity. It relies instead on a special polarity
correspondence that exists for nonnegative, positively homogeneous
convex functions that vanish at the origin; these are known as \emph{gauge
  functions}. In 2014, Friedlander, Mac\^edo, and Pong
\cite{gaugepaper} made partial progress towards connecting gauge and Lagrange dualities. In the present work, we show that gauge
duality may be regarded as a particular application of
Fenchel-Rockafellar duality theory that is different than the one
required for Lagrange duality. This connection provides a useful
vantage point from which to develop new algorithms for an important
class of convex optimization problems. We also describe how
gauge duality theory can be extended beyond the optimization of gauge
functions to the optimization of all convex functions that are bounded
below. We call this extension \emph{perspective duality}.

A convenient and fully general formulation for our approach is the
problem
\begin{alignat}{4}
	\label{eq:gauge-primal}
	&\minimize{x}  &\quad &\kappa(x)
	&\quad&\st\quad & &\rho(b-Ax) \le \sigma,
	\tag{\hbox{G$_p$}}
\end{alignat}
where $A\colon\R^n\to\R^m$ is a linear map, $b$ is an $m$-vector, and
$\kappa$ and $\rho$ are closed gauge functions. For many applications,
the function $\kappa$ is used to regularize the problem in order to
obtain solutions with certain desirable properties.  For example, in
statistical and machine-learning applications the regularizer $\kappa$
is often a nonsmooth, structure-inducing function; e.g.  the $1$-norm,
which is frequently used to encourage sparsity in the solution.  The
function $\rho$ may be regarded as a penalty function, such as the
2-norm, that measures the degree of misfit between the data $b$ and
the linear model $Ax$, and may reflect a statistical model of the
noise in the data $b$.  The perspective duality extension enables us
to consider optimization problems with a wider range of applications
by allowing functions $\kappa$ and $\rho$ that are not positively
homogenous, including the Huber function used for robust
regression~\cite{Hub}, the elastic net used for group
detection~\cite{elas_net}, and the logistic loss used for
classification~\cite{nelder1972generalized,aravkin2016level}.

The formulation \cref{eq:gauge-primal} gives rise to two different ``dual'' problems:
\begin{alignat}{4}
  \label{eq:lagrange-dual}
  &\maximize{y}  &\enspace & \ip{b}{y}   - \sigma \rho^\circ (y)
  &\enspace&\st\enspace
  & &\kappa^\circ(A\T y)\le1,
  \qquad \text{and}
  \tag{\hbox{L$_d$}}
\\\label{eq:gauge-dual}
  &\minimize{y}  &\enspace &\kappa^\circ(A\T y)
  &\enspace&\st\enspace
  & &\ip{b}{y}   - \sigma \rho^{\circ} (y) \ge 1.
  \tag{\hbox{G$_d$}}
\end{alignat}
Here $\rho^{\circ}$ and $\kappa^\circ$ are the polars of $\rho$ and
$\kappa$, which are also gauge functions; see \cref{sect: prelim} for
a precise definition. In the important case $\sigma =0$, we interpret
$\sigma \rho^{\circ}$ as the indicator function of the closure of the
domain of $\rho^\circ$ (see the discussion in \cref{sect:
  assumptions}). The first problem \cref{eq:lagrange-dual} is the
standard Lagrangian (or Fenchel-Rockafellar) dual, which is the dual
problem typically considered in connection with convex optimization
problems. Strong duality, reflected in the equality
\[
  \val\cref{eq:gauge-primal}=\val\cref{eq:lagrange-dual},
\]
and in the attainment of the optimal value of the Lagrange primal-dual
pair, holds under mild interiority conditions often referred to as the
Slater constraint qualification.  The second problem
\cref{eq:gauge-dual} is the gauge dual and is less well-known. Under
interiority conditions similar to those required by Lagrange duality,
strong duality holds in the gauge duality setting; this is reflected in the analogous equality
\[
  1=\val\cref{eq:gauge-primal}\cdot \val\cref{eq:gauge-dual},
\]
and in the attainment of the optimal value of the gauge primal-dual pair.

In certain contexts, the gauge dual \cref{eq:gauge-dual} can be
preferable for computation to the the primal \cref{eq:gauge-primal}
and the Lagrangian dual \cref{eq:lagrange-dual}, particularly when the
polar $\kappa^\circ$ has a special structure. Friedlander and
Mac\^edo~\cite{FriedlanderMacedo:2016}, for example, use gauge duality
to derive an effective algorithm for an important class of low-rank
spectral optimization problems that arise in signal-recovery
applications, including phase recovery and blind deconvolution.
Indeed, the effectiveness of numerous convex optimization
algorithms---particularly first-order methods---relies on being able to
project easily onto the constraint set. The appearance of the linear map
$A$ in the constraints of both \cref{eq:gauge-primal} and
\cref{eq:lagrange-dual} means that such methods may not be efficient, 
though some recent methods have been proposed that 
circumvent this difficulty~\cite{shefi2016dual}.  In contrast, the
map $A$ appears in the gauge dual \cref{eq:gauge-dual} only in the
objective, and computing subgradients of this objective only
requires subgradients of $\kappa^\circ$, together with the ability
to efficiently implement matrix-vector multiplication. Moreover,
typical applications occur in the regime $m\ll n$. For example, $m$ is
often logarithmic in $n$ \cite{CRT,robust_rec,tropp,don}. Because the
dual variables $y$ of \cref{eq:gauge-dual} lie in the much smaller
space $\R^m$, projections onto the feasible region may be computed
efficiently, depending on the context. An example of how an interior 
method may be used for this purpose is given in \cref{sec:glms}.

\subsection{Approach}

This paper has two main goals. The first goal, addressed in \cref{sect: gauge_sensitivity}, 
is to show how the foundations of gauge duality can be derived via a perturbation framework 
pioneered by Rockafellar\cite{rockafellar1974,rockafellar}, in which the optimal
value and optimal solution depend on parameters to the problem. We
follow Rockafellar and Wets~\cite[11.H]{rockafellarwets}, who consider
an arbitrary convex perturbation function $F$ on $\R^n\times\R^m$ that
determines how the parameters enter the problem, and define the value
functions
\begin{equation} \label{eq:general-value-functions}
p(u):=\inf_x\, F(x,u)\quad
\textt{and}\quad
q(v):=\inf_y\, F^\star(v,y).
\end{equation}
This set-up immediately yields the primal-dual pair
\begin{equation} \label{eq:fenchel-dual-pair}
  p(0)=\inf_x\, F(x,0)\quad
\textt{and}\quad p^{\star\star}(0)=\sup_y\, -F^{\star}(0,y) \equiv -q(0).
\end{equation}
Fenchel-Rockafellar duality theory flows from an appropriate choice of
$F$.  We show that gauge duality fits equally well into this framework
under a judicious choice of the perturbation function $F$, thereby
putting Fenchel-Rockafellar and gauge duality theories on an equal
footing. Strong duality, primal-dual optimality conditions, and an
interpretation of the gauge dual solutions as sensitivity
measures---i.e., subgradients of the value function---quickly follow;
cf.~\cref{seubsec:deriv_gauge}. These results, in particular, answer
an open question posed by Freund in his original work \cite{freund},
which asked for an interpretation of gauge dual variables for problems
with nonlinear constraints. It also completes a partial analysis by
Friedlander et al.~\cite{gaugepaper} on the interpretation of gauge
dual variables as sensitivity measures.

This viewpoint allows us to prove a striking relationship between
optimal solutions of the primal and optimal solutions of the
Lagrangian dual of the gauge dual: the two coincide up to scaling by
the optimal value (\cref{sect: gauge_fenchel}). Consequently,
Lagrangian primal-dual methods applied to the gauge dual can be used
to recover solutions of the original primal problem. We illustrate
this idea  in \cref{sect:numerical} with an application of Chambolle and Pock's 
primal-dual algorithm~\cite{cp} to a specific problem instance.

The second goal of this paper is to extend the applicability of the
gauge duality paradigm beyond gauges to capture more general convex
problems. \Cref{sec:perspective-duality} extends gauge duality to
problems involving convex functions that are merely nonnegative,
and by an appropriate translation, functions that are bounded from
below. The approach is based on using the perspective transform of a
convex function~\cite[p.~35]{rockafellar}, which increases a function's
domain from $\R^n$ to $\R^{n+1}$ and makes it positively homogeneous, 
enabling the property that is key to the application of
gauge duality. We term the resulting dual problem the {\em perspective
  dual}.
The perspective-polar transformation, needed to derive the perspective
dual problem, is developed in \cref{sec:perspective-duality}. Concrete 
illustrations of perspective duality for the
family of piecewise linear-quadratic functions, which are often used in
data-fitting applications, and for the setting of generalized linear models, are given in \cref{sect: casestudies}. We further explore examples of optimality conditions and primal-from-dual recovery in \cref{sect: recovery_ex}. Numerical illustrations for a
case-study of perspective duals comprise
\cref{sect:numerical}.

\section{Notation and assumptions} \label{sect: bground}

The derivation of our results relies on standard notions from convex
analysis. Unless otherwise specified, we generally follow
Rockafellar~\cite{rockafellar} for standard definitions and notation,
including domains and epigraphs, relative interiors, convex conjugate
functions, subdifferentials, polar sets, etc. In this section we
collect less well-known definitions and notation used throughout the
paper, and establish blanket assumptions on the problem data.

Let $\Rbar:=\R\cup\{ + \infty\}$ denote the extended real line, and
$\Rbar_+ := \{ x \in \Rbar\,  | x \ge 0\}$ denote the nonnegative extended reals. 
Let $f\colon \Rn \to\Rbar$ and $g \colon \R^m \to \Rbar$ denote general
closed convex functions. For a closed convex set $\Cscr \subseteq \Rn$, its convex
indicator $\delta_{\Cscr}$ is the closed convex function whose value
is zero on $\Cscr$ and $+\infty$ otherwise.  Let
$\cone\, \Cscr:= \set{ \lambda x | \lambda \ge 0,\, x \in \Cscr }$
denote the cone generated by
$\Cscr$. We often abbreviate fractions such as $(1/(2\mu))$ to $(1/2\mu)$.

\subsection{The perspective transform} \label{sect: prelim}

For any convex function $f:\R^n\to\Rbar$, its {\em perspective} is the
function on $\R^{n+1}$ whose epigraph is the cone generated by the set
$(\epi f)\times \{1\}$. %
Because this transform is not necessarily closed---even when $f$ is
closed---we choose to work with its closure, and redefine the transform as
\begin{equation}\label{def:perspective2}
  f^{\pi} (x, \lambda) :=
  \begin{cases}
    \lambda f(\lambda^{-1} x) & \hbox{if $\lambda>0$}
  \\f^{\infty} (x)            & \hbox{if $\lambda=0$}
  \\+\infty                   & \hbox{if $\lambda<0$,}
  \end{cases}
\end{equation}
where $f^{\infty} (x)$ is the {\it recession function} of $f$
\cite[Theorem 8.5]{rockafellar}.  A calculus for the perspective
transform $f\mapsto f^\pi$ is described by Aravkin, Burke, and
Friedlander~\cite[Section 3.3]{aravkin2013variational} and, for the
infinite-dimensional case, by
Combettes~\cite{Combettes2016,Combettes2017}, where properties of the
perspective transform are described in detail. We often apply more
than one transformation to a function, and in such cases, the multiple
transformations are applied in the order that they appear; e.g.,
$ f^{\pi\circ} := (f^\pi)^\circ$.

\subsection{Gauge functions}
The following is only a brief description of gauge functions. 
A complete description is given by Rockafellar~\cite[Section 15]{rockafellar}.

A convex function $\kappa: \Rn \to \Rbar$ is called a {\em gauge} if
it is nonnegative, positively homogeneous, and vanishes at the origin.
The symbols $\kappa \colon \Rn \to\Rbar$ and
$\rho \colon \R^m \to \Rbar$ will always denote closed gauges.  The
polar of a gauge $\kappa$ is the function $\kappa^\circ$ defined by
\begin{equation} \label{eq:polar-gauge-def}
 \kappa^\circ(y):= \inf\set{\mu>0 | \ip{x}{y} \le \mu \kappa(x),\ \forall x},
\end{equation}
which is also a gauge and satisfies $\kappa^{\circ\circ}=\kappa$ when $\kappa$ is closed
\cite[Theorem 15.1]{rockafellar}. For example, if $\kappa$ is a norm
then $\kappa^\circ$ is the corresponding dual norm. Note the identity
\begin{equation} \label{eq:epi-polar}
\epi\kappa^{\circ}=\set{(y,-\lambda)| (y,\lambda)\in (\epi \kappa)^{\circ}}.
\end{equation}
It follows directly from~\eqref{eq:polar-gauge-def} and positive
homogeneity of a gauge function that its polar can be characterized as
the support function to the unit level set, i.e.,
\begin{equation}\label{eq:polar-sup}
  \kappa^\circ = \delta^*_{ \Uscr_{\kappa}} = \sup \set{ \ip{u}{\cdot} | u \in \Uscr_\kappa }
\textt{where} \Uscr_{\kappa}:=\set{u | \kappa(u)\leq 1}.
\end{equation}
Moreover, $\kappa$ and $\kappa^\circ$ satisfy a H\"older-like
inequality
\begin{equation} \label{eq:holder-ineq}
  \ip x y \le \kappa(x)\cdot \kappa^\circ(y)
  \quad
  \forall x\in\dom\kappa,\ \forall y\in\dom\kappa^\circ,
\end{equation}
which we refer to as the {\em polar-gauge inequality}.
The zero level set
\[
  \Hscr_{\kappa}:=\set{u | \kappa(u)= 0}
\]
plays a key role when $\sigma =0$. It is straightforward
to show that
\begin{equation}\label{eq:unit and zero}
  \Uscr_{\kappa}^\circ=\Uscr_{\kappa^\circ}\,, \quad \Uscr_{\kappa}^\infty=\Hscr_{\kappa}\,,\quad
  (\dom \kappa)^\circ=\Hscr_{\kappa^\circ}\,,
  \quad\mbox{and}\quad
  \Hscr_{\kappa}^\circ=\cl \dom \kappa^\circ
\end{equation}
whenever $\kappa$ is closed, where $\Uscr_{\kappa}^\infty$ is the {\it
  recession cone} for $\Uscr_{\kappa}$ \cite[Section 8]{rockafellar}. We include proofs of \eqref{eq:unit and zero} in \cref{sec:sigma0_facts}.

\subsection{Assumptions on the feasible region} \label{sect: assumptions}

Define the following primal and dual feasible sets:
\begin{equation}\label{gauge feasible sets}
 \Fscr_p:=\set{u|\rho(b-u) \le \sigma} \qquad \hbox{and}\qquad
 \Fscr_d:=\set{y|\ip b y -\sigma\rho^{\circ}(y)\geq 1 }.
\end{equation}
The nonnegativity of $\rho$ implies that the Slater condition can fail
when $\sigma=0$, and thus special attention is required.  In this
case, we make the replacement
\begin{equation}\label{eq:replacement}
  (\rho,\,\sigma)\ \Rightarrow\ (\delta_{\Hscr_\rho},\,1) \textt{whenever} \sigma=0.
\end{equation}
This replacement yields a gauge optimization problem whose solution
set and optimal value coincide with those of~\eqref{eq:gauge-primal}.
Observe that because $\Hscr_\rho$ is a closed convex cone,
$\delta_{\Hscr_\rho}=\delta_{\Hscr_\rho^\circ}^*$ is a closed gauge
that satisfies, by virtue of \eqref{eq:unit and zero},
$\delta_{\Hscr_\rho}^\circ=\delta_{\Hscr_\rho^\circ}=\delta_{\cl\dom\rho^\circ}$. This
motivates the convention made immediately
following~\eqref{eq:gauge-dual} that
\begin{equation}\label{eq:convention-sigma-zero}
  \sigma \rho^{\circ} :=     \delta_{\cl\dom\rho^\circ}
                      \equiv \delta_{\Hscr_\rho^\circ}
                      \textt{when} \sigma=0.
\end{equation}
The replacement~\eqref{eq:replacement} allows us to make the useful
assumption that $\sigma>\inf \rho$, which significantly streamlines
our analysis. The convention \eqref{eq:convention-sigma-zero} also
makes sense from an epigraphical perspective, because the functions
$\sigma \rho^{\circ}$ epigraphically converge to
$\delta_{\cl\dom\rho^\circ}$ as $\sigma\downarrow 0$ \cite[Proposition 7.4(c)]{rockafellarwets}.

The gauge primal~\eqref{eq:gauge-primal} and
dual~\eqref{eq:gauge-dual} problems are said to be {\em feasible},
respectively, if the following intersections are nonempty:
\[
 A^{-1} \Fscr_p \cap(\dom \kappa) \textt{and} A\T \Fscr_d\cap ( \dom \kappa^\circ).
\]
Similarly, the primal and dual problems are said to be {\em relatively strictly
  feasible}, respectively, if the following intersections are
nonempty:
\[
  A^{-1} \left( \ri \Fscr_p \right) \cap (\ri\dom\kappa) \textt{and} A\T\ri \Fscr_d \cap (\ri \dom\kappa^\circ).
\]
If the intersections above are nonempty, with interior replacing relative interior, then we say that the problems are {\em strictly feasible}.
We have
\begin{align*}
  \ri \Fscr_p &= \begin{cases}
  \set{u | b-u \in \ri\dom \rho, \, \rho(b-u)< \sigma}  &\mbox{if $\sigma >0$}\\
   \set{u | b-u \in \ri \Hscr_{\rho}}  & \mbox{if $\sigma =0$,}
   \end{cases}
  \\
  \ri \Fscr_d &= \begin{cases}
  \set{y| y \in \ri\dom \rho^\circ, \, \ip b y -\sigma\rho^{\circ}(y) > 1 }
    &\mbox{if $\sigma >0$}\\
\set{y| y \in \ri \Hscr_{\rho}^\circ, \, \ip b y  > 1 }
    &\mbox{if $\sigma =0$,}
    \end{cases}
\end{align*}
which follows from Rockafellar~\cite[Theorem 7.6]{rockafellar} when
$\sigma>0$, and from the convention~\eqref{eq:convention-sigma-zero}
when $\sigma=0$.

We assume throughout that $\rho(b) > \sigma$. Otherwise, $\Fscr_p$
contains the origin, which is a trivial solution of
\eqref{eq:gauge-primal}. This assumption is consistent with classical
applications in signal processing and machine learning, where the
corresponding assumption is that the data $b$ does not entirely consist
of noise.
\section{Perturbation analysis for gauge duality} \label{sect: gauge_sensitivity}

Modern treatment of duality in convex optimization is based on an
interpretation of multipliers as giving sensitivity information
relative to perturbations in the problem data. No such analysis,
however, has existed for gauge duality. In this section we show that
for a particular kind of perturbation, the gauge dual
\eqref{eq:gauge-dual} can in fact be derived via such an
approach.

\subsection{General perturbation framework}

Our analysis is based on a perturbation theory described by
Rockafellar and Wets~\cite[11.H]{rockafellarwets}. In this section we
summarize the main results from \cite{rockafellarwets} that we need. Fix an arbitrary convex
function $F\colon\R^n\times\R^m\to\overline\R$, and consider the value
functions defined
by~\eqref{eq:general-value-functions}--\eqref{eq:fenchel-dual-pair}.
Observe the equality $q(0)= -p^{\star\star}(0)$.
For example, Fenchel-Rockafellar duality for the problem
\begin{equation} \label{eq:9}
  \minimize{x} \ f(Ax) + g(x),
\end{equation}
is obtained from the general perturbation theory by setting
$F(x,u)=f(Ax+u)+g(x)$. In that case, the primal-dual pair takes the
familiar form
\[
  p(0)=\inf_x \Set{f(Ax)+g(x)^{\vphantom T}}
  \textt{and}
  p^{\star\star}(0)=\sup_y \Set{-f^{\star}(-y)-g^\star(A\T y)}.
\]
Under certain conditions, described in the following theorem,
strong duality holds, i.e. $p(0)=p^{\star\star}(0)$, and the optimal values are attained.

\begin{theorem}[Multipliers and sensitivity {{\cite[Theorem 11.39]{rockafellarwets}}}]
  \label{thm:gen_dual_frame}
  Consider the primal-dual pair~\eqref{eq:fenchel-dual-pair}, where
  $F\colon\R^n\times\R^m\to\overline\R$ is proper, closed, and convex.
\begin{enumerate}[{\rm (a)}]
\item \label{it:pq_ineq} The inequality $p(0)\geq -q(0)$ always holds.
\item  \label{it:pq_eq}
If $0\in \ri\dom p$, then equality $p(0)=-q(0)$
holds and, if finite, the infimum $q(0)$ is attained with
$\partial p(0)=\argmax_y -F^{\star}(0,y)$.
Similarly, if $0\in \ri\dom q$, then equality
$p(0)=-q(0)$ holds and, if finite, the infimum $p(0)$ is attained with
$\partial q(0)=\argmin_x F(x,0)$.
\item The set $\argmax_y -F^{\star}(0,y)$ is nonempty and bounded if
  and only if $p(0)$ is finite and $0\in \inter\dom p$.
\item The set $\argmin_x F(x,0)$ is nonempty and bounded if and only
  if $q(0)$ is finite and $0\in \inter\dom q$.
\item Optimal solutions are characterized jointly through the
  conditions
  \[
    \left.
      \begin{array}{@{}l@{}}
        \bar{x}\in \argmin_x\, F(x,0)\\
        \bar{y}\in \argmax_y\, -F^{\star}(0,y)\\
        F(\bar x,0)=-F^{\star}(0,\bar y)
      \end{array}
    \right\}
    \quad\Longleftrightarrow\quad
    (0,\bar y)\in \partial F(\bar x,0)
    \quad\Longleftrightarrow\quad (\bar x,0)\in \partial F^{\star}(0,\bar y).
\]
\end{enumerate}
\end{theorem}

\begin{proof}
  The only difference between the statement of this
  theorem and that in \cite[Theorem 11.39]{rockafellarwets} is in part (b). Here, we
  make use of the relative interior rather than the interior.  Thus,
  we only prove part (b).  Suppose $0\in \ri\dom p$. If
  $p(0)=-\infty$, then $p(0)=-q(0)$ follows by
  Part~(\ref{it:pq_ineq}).  Hence we can assume that $p(0)$ is finite,
  and conclude that $p$ is proper.  By \cite[Theorem
  23.4]{rockafellar}, $\partial p(0)\ne\emptyset$, and given
  $\phi \in \partial p(0)$,
  \[
    p(0) \le p(u) - \ip{\phi}{u}
    = \inf_x  \Set{F(x,u) -
      \ip*{\begin{pmatrix} 0 \\ \phi \end{pmatrix}}
      {\begin{pmatrix} x \\u \end{pmatrix}}} \quad \forall u.
  \]
  By taking the infimum over
  $u$ and recognizing the right-hand side as
  $-F^{\star}(0,\phi)$, we deduce that $p(0) \le
  -F^\star(0,\phi)\le-q(0)$.  Combining this with Part (a) yields
  $p(0) =-F^\star(0,\phi)= -q(0)$. Hence $\phi\in \argmax_y
  -F^{\star}(0,y)\ \ne\emptyset$.  Conversely, given any $\phi
  \in \argmax_y -F^{\star}(0,y)$, we have
\[
\begin{aligned}
p(0)&=- F^\star(0,\phi)
= \inf_{x,u} \left\{F(x,u)-\ip*{\begin{pmatrix}0\\ \phi\end{pmatrix}}{\begin{pmatrix}x\\ u
\end{pmatrix}}\right\}
\\ &= \inf_{u}\left\{ p(u)-\ip{\phi}{u}\right\}
  \le p(v)-\ip{\phi}{v}\qquad\forall\ v,
\end{aligned}
\]
and so $\phi\in\partial p(0)$.
The case $0\in \ri\dom q$ follows by an analogous argument.
\end{proof}

\subsection{A perturbation for gauge duality}\label{seubsec:deriv_gauge}

We now show that the problems \cref{eq:gauge-primal} and
\cref{eq:gauge-dual} constitute a primal-dual pair under the framework
set out by \cref{thm:gen_dual_frame}. The key is to postulate the
correct pairing function $F$. In the derivation below, we show
that the gauge primal-dual pair corresponds to the primal and dual
value functions
\begin{subequations}
  \begin{align}
    \label{eq:7a}
   v_p(u)        &:=
  \inf_{\mu>0,\,x}
  \set{\mu | \rho\left(b-Ax + \mu u \right)\leq \sigma,\ \kappa(x)\leq \mu},
    \\\label{eq:7b}
    v_d(t,\theta) &:=
  \inf_{y}
  \set{\kappa^{\circ}(A\T y+t) | \langle b,y \rangle - \sigma\rho^{\circ}(y)\geq 1+\theta},
  \end{align}
\end{subequations}
where, as in \eqref{eq:gauge-dual}, we use the convention described
by~\eqref{eq:replacement} and~\eqref{eq:convention-sigma-zero}.  The
parameters $u$ and $(t,\theta)$ are perturbations to the primal and
dual gauge problems, respectively. This perturbation scheme differs
significantly from that used in Fenchel-Rockafellar
duality---cf.~\eqref{eq:9}---because of the product $\mu u$.

We begin by observing that $v_p(0)$ is equal to the optimal value of
the primal \cref{eq:gauge-primal}.  Because $u$ and $\mu$ appear as a
product in this definition, it is convenient to reparametrize the
problem by setting $\lambda:= 1/\mu$ and $w:=x/\mu$. The positive
homogeneity of $\kappa$ and $\rho$ allows us to equivalently phrase
the primal value function as
\[
  v_p(u)=\inf_{\lambda>0,\, w}\Set{1/\lambda|\rho(\lambda b - Aw +u)\leq\sigma\lambda,\ w\in \Uscr_{\kappa}}.
\]
In particular, this reparameterization shows that the value function
$v_p$ is convex because it is the infimal projection of a convex
function, and it is proper when the primal \cref{eq:gauge-primal} is
feasible.

We now construct the function $F$ appearing in
\cref{thm:gen_dual_frame} associated with this duality framework.  In
this construction, we assume that $\sigma>0$, possibly making the
replacement \eqref{eq:replacement} if $\sigma=0$.
Note that minimizing $1/\lambda$ is equivalent to minimizing
$-\lambda$ for $\lambda\ge 0$.  Define the convex function
$F\colon\R^{n}\times \R \times \R^m\to\overline\R$ by
\begin{equation*} \label{eq:5}
  F(w,\lambda,u):=
  -\lambda  %
  + \delta_{(\epi \rho)\times \Uscr_{\kappa}}
  \left(
    W
    \begin{bmatrix}w\\ \lambda\\u \end{bmatrix}
  \right),
  \ \mbox{ where }\
  W:=\begin{bmatrix}-A &b & I_{m} \\ 0 &\sigma & 0 \\ I_{n} & 0 & 0\end{bmatrix}.
\end{equation*}
Observe that the matrix $W$ is nonsingular. %

Because $(0,0,0) \in \dom F$, and $\kappa$ and $\rho$
are closed, the function $F$ is closed and proper.
This pairing function gives rise to the infimal projection problems
\begin{equation}\label{eq:10}
  p(u):=\inf_{\lambda\ge0,w}\, F(w,\lambda,u)
  \textt{and}
  q(t,\theta):=\inf_{y}\, F^\star(t,\theta,y),
\end{equation}
which correspond to the general definitions shown
in~\eqref{eq:general-value-functions}. Note that the function $p$ is
the reciprocal of $v_p$, as formalized in the following lemma (stated without proof).

\begin{lemma}\label{obs:simp}
  Equality $v_p(u)=-1/p(u)$ holds provided that $v_p(u)$ is nonzero
  and finite. Moreover, $v_p(u)=0$ if and only if $p(u)=-\infty$, and
  $p(u)=0$ if and only if $v_p(u)=+\infty$.
\end{lemma}

We now compute the conjugate of $F$, which is needed to derive the dual value
function $q$.
By Rockafellar and Wets~\cite[Theorem
11.23(b)]{rockafellarwets},
\begin{equation*} \label{eq:conjugate-of-F}
  F^{\star}(t,\theta,y)=\cl \inf_{z,\beta,r}
  \Set{\delta^{\star}_{(\epi \rho)\times \Uscr_{\kappa}}
    \begin{pmatrix}z\\ \beta\\r \end{pmatrix} |
    W\T
    \begin{bmatrix}z\\\beta\\r \end{bmatrix}
   =\begin{bmatrix}t\\\theta\\y \end{bmatrix}
   +\begin{bmatrix}0\\1\\0 \end{bmatrix}},
\end{equation*}
where the closure operation $\cl$ is applied to the function on the
right-hand side with respect to the argument $(t,\theta,y
)$. %
Using the definition of $W$, the constraint in the description of
$F^{\star}$ is precisely
$(r-A\T z, \ip b z + \sigma \beta, z) = (t, \theta +1, y)$, and the
unique vector that satisfies these constraints is
$(z,\beta,r) = (y, \, \sigma^{-1}(\theta + 1 - \ip b y ), \, t+ A\T y)$. The
closure operation is therefore superfluous, and we obtain
\begin{align*}
  F^{\star}(t,\theta,y )
  &=\delta^{\star}_{(\epi \rho)\times \Uscr_{\kappa}}
    \begin{pmatrix}y \\\sigma^{-1}(1+\theta-\langle b,y\rangle)\\ t+A\T y\end{pmatrix}
\\&=\delta^{\star}_{\epi \rho}\begin{pmatrix}y\\\sigma^{-1}(1+\theta-\ip b y)\end{pmatrix}
  +\delta^{\star}_{\Uscr_{\kappa}}(t+A\T y).
\end{align*}
Since
$\delta^{\star}_{\epi
  \rho}(z_1,z_2)=\delta_{\epi\rho^\circ}(z_1,-z_2)$
and
$\delta^\star_{\Uscr_\kappa}=\kappa^\circ$ by \eqref{eq:epi-polar} and \eqref{eq:polar-sup}, %
this reduces to
\[
  F^{\star}(t,\theta,y ) =\delta_{\epi \rho^{\circ}}\begin{pmatrix}y
    \\-\sigma^{-1}(1+\theta-\ip b y)\end{pmatrix}
  +\kappa^{\circ}(t+A\T y).
\]
The application of \cref{thm:gen_dual_frame} asks that we evaluate these conjugates at
$(t,\theta)=(0,0)$, which yields the expression
\begin{equation*}
F^{\star}(0,0,y)=      \left\{\begin{array}{ll}
      \kappa^{\circ}(A\T y) & \textrm{if }~ \langle b,y\rangle-\sigma\rho^\circ(y)\geq 1 \\
      +\infty & \mbox{otherwise.}
    \end{array}
\right.
\end{equation*}
Thus, the dual problem
\[
  -q(0,0) = -\inf_y\, F^\star(0,0,y) = \sup_{y}\,-F^{\star}(0,0,y)
\]
recovers, up to a sign change, the required gauge dual problem
\eqref{eq:gauge-dual} when $\sigma > 0$.  When $\sigma=0$, we also
recover the gauge dual problem \eqref{eq:gauge-dual} by making the
appropriate substitutions~\eqref{eq:replacement} under the
convention~\eqref{eq:convention-sigma-zero}. %

This discussion justifies the definition of the dual perturbation
function $v_d(t,\theta) :=\inf_{y}\, F^{\star}(t,\theta,y)$, which is
equivalent to the expression~\eqref{eq:7b}.  Note that $v_d(0,0)$ is
the optimal value of \eqref{eq:gauge-dual}. In summary, $(-1/v_p)$ and
$v_d$, respectively, play the roles of $p$ and $q$ as defined
in~\eqref{eq:10}. In the application of \cref{thm:gen_dual_frame}, we
identify $x$ with $(w,\lambda)$, and $v$ with $(t,\theta)$.

\subsection{Proof of gauge duality} \label{sec:proof-guage-duality}

We now use the perturbation framework from \cref{seubsec:deriv_gauge}
to prove weak and strong duality results for the gauge duality
setting. \cref{thm:gauge-duality}~\cite[section 5]{gaugepaper} is
already known, but the proof via perturbation is new.

The following auxiliary
result ties the feasibility of the gauge pair \eqref{eq:gauge-primal}
and \eqref{eq:gauge-dual} to the domain of the value function. The
proof of this result, which is largely an application of the calculus
of relative interiors, is deferred to \cref{sec:proof-crefpr}.

\begin{lemma}[Feasibility and domain of the value function]\label{prop:strict_feas}
  If the primal \eqref{eq:gauge-primal} is relatively strictly
  feasible, then $0\in \ri\dom p$. If the dual
  \eqref{eq:gauge-dual} is relatively strictly feasible, then
  $0\in \ri\dom v_d$. The analogous implications, where the $\ri$
  operator is replaced by the $\inter$ operator, hold under strict
  feasibility (not relative).
\end{lemma}

The duality relations in the gauge framework follow analogous
principles to Lagrange duality, except that instead of an additive
relationship between the primal and dual optimal values
the relationship is multiplicative. The following theorem
summarizes weak and strong duality for gauge optimization.

\begin{theorem}[Gauge duality \cite{gaugepaper}]
  \label{thm:gauge-duality}
  Set $\nu_p:= v_p(0) $ and $\nu_d:= v_d(0,0)$.
  Then the following relationships hold for the
  gauge primal-dual pair \eqref{eq:gauge-primal} and
  \eqref{eq:gauge-dual}.
  \begin{enumerate}[{(a)}]
  \item \label{it:gauge-duality-basic} (Basic Inequalities)
  It is always the case that
\[
  \text{(i)}\quad(1/\nu_p)\leq \nu_d
  \quad\text{and}\quad
  \text{(ii)}\quad(1/\nu_d)\le \nu_p.
\]
In particular, if $\nu_p=0$ (resp.~$\nu_d=0$), then
\eqref{eq:gauge-dual} (resp.~\eqref{eq:gauge-primal}) is infeasible.

  \item \label{it:gauge-duality-weak} (Weak duality)  If $x$ and $y$
    are primal and dual feasible, then
    \[
      1\le \nu_p \nu_d\le\kappa(x)\cdot \kappa^{\circ}(A\T y).
    \]
  \item \label{it:gauge-duality-strong} (Strong duality)  If the dual
    (resp. primal) is feasible and the primal (resp. dual) is
    relatively strictly feasible, then $\nu_p\nu_d = 1$ and the gauge
    dual (resp. primal) attains its optimal value.
  \end{enumerate}
\end{theorem}

\begin{proof}
To simplify notation, in this proof we denote the optimal value of the primal value function by
$p_0\equiv p(0)$.

Part (a). We begin with the inequality (\textit{i}). \cref{thm:gen_dual_frame}
guarantees the inequality
\begin{equation}\label{eqn:main_ineq}
  p_0\geq -\inf_{y} F^{\star}(0,0,y)=-\nu_d.
\end{equation}
By \cref{obs:simp}, whenever $\nu_p$ is nonzero and finite, equality
$p_0=-1/\nu_p$ holds, which together with~\eqref{eqn:main_ineq} yields
 (\textit{i}).  If, on the other hand, $\nu_p=+\infty$, then
(\textit{i}) is trivial. Finally, if $\nu_p=0$, \cref{obs:simp}
yields $p_0=-\infty$, and hence \eqref{eqn:main_ineq} implies
$\nu_d=+\infty$, and (\textit{i}) again holds. Thus, (\textit{i})
holds always.  To establish (\textit{ii}), it suffices to consider
the case $\nu_d=0$. From \eqref{eqn:main_ineq} we conclude $p_0\geq 0$, that is either
$p_0=0$ or $p_0=+\infty$.  By \cref{obs:simp}, the first case $p_0=0$ implies
$\nu_p=+\infty$ and therefore (\textit{ii}) holds. The second case $p_0=+\infty$ implies that the primal problem is infeasible, that is $\nu_p=+\infty$, and again (\textit{ii}) holds. Thus (\textit{ii}) holds always, as required.

Part (b).  Because the gauge primal and dual problems are both
feasible, $\nu_p$ and $\nu_d$ are nonzero and finite so the result
follows from part (a).

Part (c). Suppose the dual is feasible and the primal is
relatively strictly feasible.  In particular, both $\nu_p$ and
$\nu_d$ are nonzero and finite by part (a).
Hence $1\leq \nu_p\nu_d=-\nu_d/p_0$.
On the other hand, by \cref{prop:strict_feas} the assumption that the primal is relatively
strictly feasible implies $0\in \ri\dom p$. This last inequality implies $p_0=p(0)$ is
finite, and hence $p(\cdot)$ is proper.
\cref{thm:gen_dual_frame}(b) tells us that $p_0=-\nu_d$ and the
infimum in the dual $\nu_d$ is attained. Thus we deduce
$1= \nu_p \nu_d$, as claimed.

Conversely, suppose that the primal is feasible and the dual is
relatively strictly feasible.
Then, by \cref{prop:strict_feas}, $0\in \ri\dom q$ . This in turn implies
$p_0=-\nu_d$ and that the infimum in $p(0)$ is attained. Since the primal is
feasible, by \cref{obs:simp}, $p_0$ is nonzero, and hence $1= \nu_p \nu_d$
and the infimum in the primal is attained.
\end{proof}

\subsection{Gauge optimality conditions}

Our perturbation framework can be harnessed to develop optimality
conditions for the gauge pair that relate the primal-dual solutions to
subgradients of the corresponding value function. This yields a
version of parts (b) and (d) in \cref{thm:gen_dual_frame} that are
specialized to gauge duality.

\begin{theorem}[Gauge multipliers and
  sensitivity] \label{thm:gauge_opt_pert} The following relationships
  hold for the gauge primal-dual pair \eqref{eq:gauge-primal} and
  \eqref{eq:gauge-dual}.
  \begin{enumerate}[(a)]
  \item \label{it:du_solns_gauge}
    If the primal is relatively strictly feasible
    and the dual is feasible, then the set of optimal solutions for
    the dual is nonempty and coincides with
    \[
      \partial p(0) = \partial({-}1/v_p)(0).
    \]
    If it is further assumed that the primal is strictly feasible, then the
    set of optimal solutions to the dual is bounded.

  \item \label{it:pr_solns_gauge}
    If the dual is relatively strictly feasible and
    the primal is feasible, then the set of optimal solutions for the primal
    is nonempty with solutions
    $x^* = w^*/\lambda^*$, where
    \[
      (w^*, \lambda^*) \in \partial v_d(0,0)
      \textt{and}
      \lambda^*>0.
    \]
    If it is further assumed that the dual is strictly feasible, then the set of optimal solutions to the primal is bounded.
\end{enumerate}
\end{theorem}
\begin{proof}
  Part (a). Because~\eqref{eq:gauge-primal} is relatively strictly
  feasible, it follows from \cref{prop:strict_feas} that
  $0\in\ri\dom p$, and because the dual is feasible, $p(0)$ is
  finite. \cref{thm:gen_dual_frame} and \cref{obs:simp} then imply the
  conclusion of Part~(a).  The statement on the boundedness of the set
  of the optimal solutions to the dual follows from
  \cref{thm:gen_dual_frame}.

  Part (b). Because \eqref{eq:gauge-dual} is relatively strictly
  feasible, it follows from \cref{prop:strict_feas} that
  $0\in\ri\dom v_d$, and because the primal is feasible, $v_d(0)$ is
  finite. \cref{thm:gen_dual_frame} then implies that the optimal
  primal set is nonempty, and
  $\argmin_{w,\lambda}\, F(w,\lambda, 0) =\partial v_d(0,0)$.  Because
  the primal and dual problems are feasible, any pair
  $(w^*,\lambda^*) \in \argmin_{w,\lambda} F(w,\lambda, 0)$ must
  satisfy $\lambda^* >0$ by \cref{thm:gauge-duality} and
  \cref{obs:simp}. Thus, this inclusion is equivalent to
  $x^* = w^*/\lambda^*$ being optimal for the primal problem, with
  optimal value $1/\lambda^*$. This proves Part~(b).  The statement on
  the boundedness of the set of the optimal solutions to the primal
  again follows from \cref{thm:gen_dual_frame}.
\end{proof}

We use the sensitivity interpretation given by \cref{thm:gauge_opt_pert}  to develop a set of necessary and sufficient optimality conditions that mirror the more familiar KKT conditions
from Lagrange duality.
For a primal-dual optimal pair $(x^*, y^*)$, the condition $\rho^\circ(y^*)=0$
characterizes a degenerate case when $\sigma>0$ because in that case the primal
constraint is inactive at $x^*$ (i.e., $\rho(b-Ax^*) < \sigma$). On the other
hand, the dual constraint is always active at optimality because the positive
homogeneity of the dual objective and the dual constraint imply $\ip{b}{y^*} -
\sigma\rho^\circ(y^*) = 1$. The full primal-dual
optimality conditions for gauge duality are described in the following theorem.

\begin{theorem}[Optimality conditions]
  \label{thm:gauge-optimality-conds}
  Suppose both problems of the gauge dual pair \cref{eq:gauge-primal} and \cref{eq:gauge-dual} are relatively strictly feasible,
  and the pair $(x^*, y^*)$ is primal-dual
  feasible. Then $(x^*, y^*)$ is primal-dual optimal if and only if it
  satisfies the conditions
  \begin{subequations}
    \begin{align}
        \rho(b-Ax^*) &= \sigma \quad \text{or}  \quad \rho^\circ(y^*) = 0 & &\hbox{(primal activity)}  \label{eq:pr-feas}
      \\ \ip{b}{y^*} - \sigma\rho^\circ(y^*) &= 1 &&\hbox{(dual activity)} \label{eq:du-feas}
      \\ \ip{x^*}{A\T y^*} &= \kappa(x^*)\cdot\kappa^\circ(A\T y^*) &&\hbox{(objective alignment)} \label{eq:objalign}
\\ \ip{b-Ax^*}{y^*} &= \sigma\rho^\circ(y^*). &&\hbox{(constraint alignment)}
\label{eq:consalign}
    \end{align}
  \end{subequations}
\end{theorem}

\begin{proof}
  First suppose that $(\bar x,\bar y)$ satisfies
  \eqref{eq:pr-feas}-\eqref{eq:consalign}.
  By \cref{thm:gauge-duality}, to show that $(\bar x,\bar y)$ is
  primal-dual optimal it is sufficient to show that
  $\kappa(\bar x) \cdot \kappa^{\circ}(A\T\bar y)=1$.
  Add~\eqref{eq:objalign} and \eqref{eq:consalign} to obtain
  \[ \ip{b}{\bar y} = \kappa(\bar x) \cdot \kappa^{\circ}(A\T \bar y)
    +\sigma\rho^\circ(\bar y).
  \]
  By combining the above with \eqref{eq:du-feas} we obtain
  $\kappa(\bar x) \cdot \kappa^{\circ}(A\T\bar y)=1$, as desired.

  Suppose now that $(x^*, y^*)$ is primal-dual optimal.
  We begin by assuming that $\sigma>0$
  and obtain the case $\sigma=0$ by applying the result for the $\sigma>0$
  case under
  the replacement \eqref{eq:replacement}.
  By the
  positive homogeneity of $\kappa^\circ$ and the optimality of $y^*$,
  \eqref{eq:du-feas} holds.  Also note that $\kappa(x^*)$ and
  $\kappa^\circ(A\T y^*)$ are both nonzero and finite because of the
  strong duality guaranteed by \cref{thm:gauge-duality}.

Define $\lambda^* := 1/\kappa(x^*)$ and $w^*:=\lambda^*x^*,$ so that $\kappa(w^*)=1$. By \cref{thm:gen_dual_frame}(e) and \cref{thm:gauge_opt_pert}(b),
we must have
  $(0,0,y^*)\in \partial F(w^*, \lambda^*,0)$. %
  Since the primal problem is relatively strictly feasible, we can apply \cite[Theorem 23.9]{rockafellar}
  to deduce the characterization
  \begin{equation}
    \label{eq:subdiff}
    \partial F(w,\lambda, 0) =
    -{\begin{bmatrix} 0 \\ 1 \\ 0 \end{bmatrix}} +W\T \ncone_{(\epi
      \rho)\times \Uscr_k}
    \begin{pmatrix} \lambda b - Aw   \\ \sigma \lambda \\ w \end{pmatrix},
  \end{equation}
  where $\ncone_\Cscr(\cdot)$ denotes the normal cone to a set
  $\Cscr$.  We now consider two cases. First, suppose
  $\rho(\lambda^* b-A w^*)=\lambda^* \sigma.$ Then~\eqref{eq:pr-feas}
  holds, and by straightforward computations involving
  only~\eqref{eq:polar-sup} and the definitions of normal cones and
  subdifferentials, we have
  \begin{equation*}
   \ncone_{(\epi \rho)\times \Uscr_k}
    \begin{pmatrix} \lambda^* b - Aw^*   \\ \sigma \lambda^* \\ w^* \end{pmatrix}
    = \ncone_{\epi \rho}
      \begin{pmatrix} \lambda^* b - Aw^*   \\ \sigma \lambda^* \end{pmatrix} \times \ncone_{\Uscr_k}(w^*),
    \end{equation*}
    where
    \begin{equation*}
    \ncone_{\epi \rho}
    \begin{pmatrix} \lambda^* b - Aw^*   \\ \sigma \lambda^* \end{pmatrix}
     = \cone\left(\partial \rho (\lambda^* b - Aw^*)  \times \{-1\}\right)
    \end{equation*}
    and
    $ \ncone_{\Uscr_k}(w^*)= \set{ v| \kappa^\circ (v) \le \ip v {w^*}
    }$.  Substitute these formulas into \eqref{eq:subdiff} to obtain
    \begin{equation*}\label{eq:F sd}
      \partial F(w^*,\lambda^*, 0)
      =
      \Set{ \begin{bmatrix} v- \mu A\T z  \\ \mu(\ip{b}{z}-\sigma) - 1\\ \mu z \end{bmatrix} |
        \mu \ge 0, \;
        z \in \partial \rho (\lambda^* b - Aw^*), \;
        v\in\ncone_{\Uscr_\kappa}(w^*)
      }.
  \end{equation*}
  We deduce the existence of $z^* \in \partial \rho (\lambda^* b - Aw^*)$
  and $ \mu^*\ge 0$ such that
  \begin{subequations}
    \begin{align}
      y^*&= \mu^* z^* \label{eq:subdiff2} \\
      \mu^* (\ip {b }{z^*}  - \sigma) &=1 \label{eq:subdiff1} \\
      \kappa^\circ(\mu^* A\T z^*) &\le \ip{\mu^* A\T z^*}{w^*}. \label{eq:subdiff3}
    \end{align}
  \end{subequations}
  Note that $\mu^*=0$ cannot satisfy \eqref{eq:subdiff1}, hence
  \eqref{eq:subdiff3}, together with the polar-gauge inequality and
  the fact that $\kappa(w^*)=1$, implies
  \[
    \kappa^\circ(A\T y^*)\cdot \kappa(w^*) = \kappa^\circ(A\T y^*) \le \ip{A\T
      y^*}{w^*}\le \kappa^\circ(A\T y^*)\cdot \kappa(w^*).
  \]
  Equality must hold in the above, and dividing through by $\lambda^*>0$
  we see that \eqref{eq:objalign} is satisfied.  Finally, we aim to
  show that \eqref{eq:consalign} holds using the fact that
  $y^* \in \mu^* \partial \rho (\lambda^* b - Aw^*)$.  From the
  characterization~\eqref{eq:polar-sup} of the polar, we have
  \begin{equation}
    \label{eq: polar subdiff}
    \partial \rho (u)=\argmax_y\set{\ip {y}{u} | \rho^\circ(y)\le1}.
  \end{equation}
  In particular, this characterization implies
  $\ip{y^*/\mu^*}{\lambda^* b - Aw^*} \ge \ip{0}{\lambda^* b - Aw^*} =
  0.$ If $\rho(\lambda^* b - Aw^*)=0$, then by the polar-gauge
  inequality~\eqref{eq:holder-ineq} we have
  \[
    0 \le \ip {y^*}{\lambda^* b - A w^*} \le \rho(\lambda^* b - Aw^*)
    \cdot \rho^\circ(y^*) = 0,
  \]
  which gives condition~\eqref{eq:consalign} after dividing through by
  $\lambda^*$. On the other hand, if $\rho(u)>0$ then the
  set~\eqref{eq: polar subdiff} is given by
  $\set{y | \rho(u)=\ip{y}{u}, \, \rho^\circ(y)=1}$.
  Thus when $\rho(\lambda^* b - Aw^*)>0$, we again have
  $\ip {y^*/\mu^*}{\lambda^* b - A w^*} = \rho^\circ(y^*/\mu^*) \cdot
  \rho(\lambda^* b - Aw^*)$, and multiplying through by
  $\mu^*/\lambda^*$ and applying~\eqref{eq:pr-feas} gives~\eqref{eq:consalign}.

  We have shown the forward implication of the theorem when
  $\rho(\lambda^* b-A w^*)=\lambda^* \sigma.$ The other case we need
  to consider is when $\rho(\lambda^* b-A w^*)<\lambda^* \sigma,$ or
  equivalently when $\rho( b-A x^*)<\sigma$. An easy
  argument (e.g., see \cite[Proposition 2.14(iv)]{clarke_direct}) shows
\[ \ncone_{\epi \rho}(\lambda^*b-Aw^*, \lambda^*\sigma) = \ncone_{\dom \rho}(\lambda^*b-Aw^*) \times \{0\}. \]
Similar to the first case, we now have
\begin{equation*}\label{eq:F sd2}
  \partial F(w^*,\lambda^*, 0)
    =
    \Set{ \begin{bmatrix} v- A\T z  \\  \ip{b}{z} - 1\\  z \end{bmatrix} |
    z \in \ncone_{\dom \rho}(\lambda^*b-Aw^*), \;
    v\in\Nscr_{\Uscr_\kappa}(w^*)
    }.
  \end{equation*}
  We deduce that $y^* \in \ncone_{\dom \rho}(\lambda^*b-Aw^*)$ and
  also that $\ip b {y^*} =1$ and
  $\kappa^\circ(A\T y^*) \le \ip{ A\T y^*}{w^*}$. Again, because
  $\kappa(w^*)=1$, the polar-gauge inequality implies
  \eqref{eq:objalign} holds.

  We now show that $\rho^\circ(y^*) = 0$ and
  $\ip {b-Ax^*} {y^*} = 0$, which, if true, establishes \eqref{eq:pr-feas} and
  \eqref{eq:consalign} are satisfied as well.  First note
  that $y^* \in \ncone_{\dom \rho}(u)$ implies
  $y^* \in (\dom \rho)^\circ$, which  %
  implies
  $\rho^\circ(y^*)=0$ by \cref{eq:unit and zero} .
  Thus, by \eqref{eq:du-feas},
  \eqref{eq:objalign}, and the fact that
  $\kappa(x^*)\cdot\kappa^\circ(A\T y^*) =1$ from \cref{thm:gauge-duality},
  we have
  \[ \ip {b-Ax^*} { y^*} = 1- \kappa(x^*)\cdot\kappa^\circ(A\T y^*) =
    0. \] Thus if $(x^*,y^*)$ is primal-dual optimal, then
  \eqref{eq:pr-feas}-\eqref{eq:consalign} hold, as claimed.
  This finishes the proof for $\sigma>0$.

  Let us now consider the case when $\sigma=0$ and apply what we have
  just proved to the pair \eqref{eq:gauge-primal} and
  \eqref{eq:gauge-dual} under the replacement
  \eqref{eq:replacement}. Then $(x^*, y^*)$ is primal-dual optimal if
  and only if the conditions \eqref{eq:pr-feas}-\eqref{eq:consalign}
  hold with $(\rho,\sigma)=(\delta_{\Hscr_\rho},1)$, i.e.,
\begin{alignat*}{2}
\delta_{\Hscr_\rho^\circ}(y^*)&=0, &\qquad
\ip{x^*}{A\T y^*}&=\kappa(x^*)\cdot\kappa^\circ(A\T y^*),
\\
\ip{b}{y^*}-\delta_{\Hscr_\rho^\circ}(y^*)&=1, &\qquad
\ip{b-Ax^*}{y^*}&=0.
\end{alignat*}
If we combine this with primal feasibility, $\rho(b-Ax^*)=0$, and use the identity \eqref{eq:convention-sigma-zero} that $0\cdot\rho^\circ=\delta_{\Hscr_\rho^\circ}$ , then these conditions are
equivalent to  \eqref{eq:pr-feas}-\eqref{eq:consalign} for $\sigma=0$, $\rho$, and $\rho^\circ$ as written above.
\end{proof}

The following corollary describes a variation of the optimality
conditions outlined by \cref{thm:gauge-optimality-conds}. These
conditions assume that a solution $y^*$ of the dual problem is
available, and gives  conditions that can be used to determine
a corresponding solution of the primal problem.   An application of the following
result appears in \cref{sect:
  recovery_ex}.

\begin{corollary}[Gauge primal-dual recovery] \label{cor:recovery}
  Suppose that the primal-dual pair~\eqref{eq:gauge-primal}
  and~\eqref{eq:gauge-dual} are each relatively strictly feasible.
  If $y^*$ is
  optimal for \eqref{eq:gauge-dual}, then for any primal feasible
  $x$ the following conditions are equivalent:
\begin{enumerate}[{(a)}]
\item $x$ is optimal for \eqref{eq:gauge-primal};
\item $ \ip{x}{A\T y^*} = \kappa(x)\cdot\kappa^\circ(A\T y^*)$ and
$b-Ax  \in \partial(\sigma\rho^\circ)(y^*)$;
\item $A\T y^* \in \kappa^\circ(A\T y^*) \cdot \partial \kappa (x)$ and
$b-Ax  \in \partial(\sigma\rho^\circ)(y^*)$,
\end{enumerate}
where, by convention, $\sigma\rho^\circ=\delta_{\cl\dom\rho^\circ}$ when $\sigma=0$.
\end{corollary}
\begin{proof}
  We use the optimality conditions given in
  \cref{thm:gauge-optimality-conds}. As noted before, by the optimality of $y^*$ we automatically have equality~\eqref{eq:du-feas} in the
  dual constraint.

  We first show that (b) implies (a).
  Suppose (b) holds. Then~\eqref{eq:objalign} holds automatically.
  From the characterization~\eqref{eq:polar-sup} of the polar, we have
  \begin{equation}\label{eq:max-element-rho}
    \sigma\rho^\circ(y^*) =
    \left.\begin{cases}
    \displaystyle\sigma\cdot\sup_{\mathclap{\rho(z)\le1}}\ \ip {y^*} z
     & \hbox{if  $\sigma >0$}\\
    \delta_{\cl\dom\rho^\circ}(y^*)
    & \hbox{if $\sigma =0$}
    \end{cases}
    \right\}
   = \sup_{\rho(z)\le\sigma}\ip {y^*} z,
  \end{equation}
  where the case $\sigma=0$ uses the
  convention~\eqref{eq:convention-sigma-zero}.  Thus,
  $\partial (\sigma \rho^\circ) (y^*)$ is the set of maximizing
  elements in this supremum.  Because
  $b-Ax \in \partial(\sigma \rho^\circ)(y^*)$, it holds that
  $\rho(b-Ax)\le\sigma$. If we additionally use the polar-gauge
  inequality, we deduce that
  \[
    \sigma\rho^\circ(y^*)
    =   \ip{y^*}{b-Ax}
    \le \rho(b-Ax)\cdot\rho^\circ(y^*)
    \le\sigma\rho^\circ(y^*),
  \]
  and therefore the above inequalities are all tight.  Thus
  conditions~\eqref{eq:pr-feas} and~\eqref{eq:consalign} hold, and by
  \cref{thm:gauge-optimality-conds}, $(x,y)$ is a primal-dual optimal
  pair.

  We next show that (a) implies (b). Suppose that $x$ is optimal for
  \eqref{eq:gauge-primal}. Then the first condition of (b) holds by
  (\ref{eq:objalign}), and (\ref{eq:pr-feas}) and (\ref{eq:consalign})
  combine to give us
  \[
    \sigma \rho^\circ (y^*) = \rho(b-Ax)\cdot\rho^\circ(y^*) = \ip{b-Ax}{y^*}.
  \]
  This implies that $z:=b-Ax$ is a maximizing element of the supremum
  in~\eqref{eq:max-element-rho}, and thus
  $b-Ax \in  \partial(\sigma\rho^\circ) (y^*).$

  Finally, to show the equivalence of (b)
  and (c), note that by the polar-gauge inequality,
  $ \ip{x}{A\T y^*} = \kappa(x)\cdot\kappa^\circ(A\T y^*)$ if and only if
  $x$ minimizes the convex function
  $\kappa^\circ(A\T y^*)\,\kappa(\cdot) -\ip{\cdot}{A\T y^*}.$ This, in turn, is true if and only if
  $0 \in \kappa^\circ(A\T y^*)\, \partial \kappa(x) - A\T y^*$,
  or equivalently, $A\T y^* \in \kappa^\circ(A\T y^*) \cdot \partial \kappa (x)$.
\end{proof}

\subsection{The relationship between Lagrange and gauge multipliers} \label{sect: gauge_fenchel}

We now use the perturbation framework for duality to establish a
relationship between gauge dual and Lagrange dual variables. We begin
with an auxiliary result that characterizes the subdifferential of the
perspective function \eqref{def:perspective2}.
Combettes \cite[Prop.~2.3(v)]{Combettes2016} also describes an equivalent formula for the subdifferential, though the derivation and subsequent form of the expression are very different. The formula in \cref{lemma:subdiff} is more suitable for our purposes.

\begin{lemma}[Subdifferential of perspective
  function] \label{lemma:subdiff} Let $g: \Rn \to \Rbar$ be a closed
  proper convex function. Then for $(x,\mu)\in\dom g^\pi$, equality holds:
  \[
    \sd g^\pi(x,\mu)=
    \begin{cases}
      \set{(z,-g^\star(z))\,\mid\, z\in\sd g(x/\mu)}&\mbox{if $\mu>0$}\\
      \set{(z,\gam)\,\mid\, (z,-\gam)\in\epi g^\star,\ z\in\sd g^\infty(x)}&\mbox{if $\mu=0$.}
    \end{cases}
  \]
\end{lemma}

\begin{proof}
  Recall that the subdifferential of the support function to any
  nonempty closed convex set $\Cscr$ is given by
  $ \sd \delta_\Cscr^\star (x)=\argmax\set{\ip{z}{x}|z\in \Cscr}$
  \cite[Theorem 23.5 and Corollary 23.5.3]{rockafellar}.  By
  \cite[Corollary 13.5.1]{rockafellar}, $g^\pi = \delta_\Cscr^{\star}$,
  where $ \Cscr = \set{ (z,\gamma) | g^\star(z) \le -\gamma} $ is a
  closed convex set. If $(x,\mu)\in\dom g^\pi$, then $\Cscr$ is
  nonempty and
  \begin{equation*}
      \sd g^\pi(x,\mu)=
      \argmax_{(z,\gamma)\in\Cscr}\set{\ip{(x,\mu)}{(z,\gam)}}
      =
      \argmax_{(z,\gamma)\in\Cscr}\set{\ip{z}{x}+\mu\gam}.
  \end{equation*}
  Suppose now that $\mu >0$. Then
  \begin{equation}\label{eq:1}
    \sup_{(z,\gamma)\in\Cscr}\set{\ip{z}{x}+\mu\gam}
    \ =\ 
    \sup_{\mathclap{z\in\dom{g^\star}}}\,\set{\ip{z}{x}-\mu g^\star(z)}
    =
    \mu\cdot\sup_{\mathclap{z\in\dom{g^\star}}}\,\set{\ip{z}{x/\mu}-g^\star(z)}.
    \end{equation}
Using the expression for the subdifferential of a support function,
$(z,\gam)$ achieves the supremum of~\eqref{eq:1}
if $z\in\sd g(x/\mu)$ and $-\gam =g^\star(z)$.  On the other hand, if
$\mu=0$ then
\[
  \sup_{(z,\gamma)\in\Cscr}\,\ip{z}{x}
  =\ \sup_{\mathclap{z\in\dom{g^\star}}}\ \ip{z}{x}=\delta^{\star}_{\dom{g^\star}}(x)
=g^\infty(x).
\]
Again using the expression for the subdifferential of a support
function, $(z,\gam)$ achieves the supremum of~\eqref{eq:1} if
and only if $z\in\sd g^\infty(x)$ and $(z,-\gam)\in\epi{g^\star}$.
\end{proof}

We now state the main result relating the optimal solutions of \eqref{eq:gauge-primal} to the optimal solutions of the Lagrange dual of \eqref{eq:gauge-dual}.

\begin{theorem} \label{thm:lagrange} Suppose that the gauge dual
  \eqref{eq:gauge-dual} is relatively strictly feasible and the primal
  \eqref{eq:gauge-primal} is feasible. Let $\lagrange$ denote the
  Lagrange dual of~\eqref{eq:gauge-dual}, and let $\nu\L$ denote its
  optimal value.  Then
\[
  z^* \text{ is optimal for } \lagrange
  \ \iff\
  \mbox{$z^*/\nu\L$  is optimal for \eqref{eq:gauge-primal}}.
  \]
\end{theorem}

\begin{proof}
  We first note that $\lagrange$ can be derived via the framework of
  \cref{thm:gen_dual_frame} through the Lagrangian value function
  \[
    h (w) = \inf_y \Set{ \kappa^\circ (A\T y + w) +
      \delta_{\ip{b}{\cdot} -\sigma \rho^\circ (\cdot) \ge 1} (y)}.
  \] Here $h$ plays the role of $p$ in \cref{thm:gen_dual_frame};
  cf. \cite[Example 11.41]{rockafellarwets}.  
  Strong duality in~\cref{thm:gauge-duality}
  guarantees that $h(0)$ is nonzero and finite, and by
  \cref{prop:strict_feas},
  \[
    \eqref{eq:gauge-dual} \text{ relatively strictly feasible } \implies (0,0) \in
    \ri\dom v_d \implies 0 \in \ri\dom h.
  \]
  Thus, it follows from \cref{thm:gen_dual_frame} that the optimal
  points $z^*$ for $\lagrange$ are characterized by
  $z^* \in \partial h(0)$. Note also that $h(0) = \nu_L$.

  On the other hand, by \cref{thm:gauge_opt_pert}(b) the solutions to
  \eqref{eq:gauge-primal} are precisely the points
  $w^*/\lambda^*$ such that
  $(w^* ,\lambda^*) \in \partial v_d(0,0)$. Thus to relate the
  solution sets of $\lagrange$ and \eqref{eq:gauge-primal}, we must
  relate $\partial h(0)$ and $\partial v_d(0,0)$.

  For $\theta$ in a neighborhood of zero and all $t$, by positive homogeneity of
  $\kappa^\circ$ and $\rho^\circ$ we have
  \[ v_d(t,\theta) = (1+\theta)h\left(\frac{t}{1+\theta}\right) =
    \inf_y \left\{ (1+\theta) \kappa^\circ \left(A\T y +
        \frac{t}{1+\theta}\right) + \delta_{\ip{b}{\cdot} -\sigma
        \rho^\circ (\cdot) \ge 1} (y) \right\} .
  \]
  Thus by \cref{lemma:subdiff},
  $ \partial v_d(0,0) = \set{ (z, -h^\star(z)) | z \in \partial
    h(0)}.$ However, for $z \in \partial h(0)$ the Fenchel-Young
  equality gives us
  \[ 0 = \ip 0 z = h^\star(z) + h(0) = h^\star(z) + \nu_L. \] Thus we
  obtain the convenient description
  \[ \partial v_d(0,0) = \partial h(0) \times \{ h(0)\} = \partial
    h(0) \times \{ \nu_L\} \] and the set of optimal solutions for
  \eqref{eq:gauge-primal} is precisely
  $\frac{1}{\nu_L} \partial h(0)$.
\end{proof}

\section{Perspective duality}\label{sec:perspective-duality}

We now move on to an extension of the gauge duality framework, which
allows us to consider functions that are not necessarily positively
homogeneous, but continue to be nonnegative and convex. 
(The same framework applies to functions that are bounded
below because these can be made nonnegative by translation.) For
the remainder of the paper, consider functions $f:\Rn\to\Rbar_+$ and
$g:\R^m\to\Rbar_+$, that are closed, convex and nonnegative over their
domains. In this section we derive and analyze the
\emph{perspective-dual} pair
\begin{alignat}{4}
  \label{eqn:target}
  &\minimize{x} &\quad &f(x) &\quad&\st\quad && g(b-Ax)\le\sigma,
  \tag{\hbox{N$_p$}}
\\
  \label{eqn:persp_dual}
  &\minimize{y,\ \alpha, \ \mu}
  &\quad& f^\pp( A\T y , \alpha)
 &\quad&\st\quad
 &&\ip b y  - \sigma g^\pp (y, \mu) \ge 1 - (\alpha+\mu).
 \tag{\hbox{N$_d$}}
\end{alignat}
The functions $f^\pp$ and $g^\pp$ are the polars of the perspective
transforms of $f$ and $g$. This transform is a key operation needed to
derive perspective duality. In the next section we describe properties
of that transform and its application to the derivation of the
perspective-dual pair. Throughout this section, we assume that
$\sigma>\inf_u g(u)\ge 0$.

\subsection{Perspective-polar transform}

Given a closed proper convex function $f:\Rn\to\Rbar_+$, define the
perspective-polar transform by $f^\pp:=(f^{\pi})^\circ$.%

An explicit characterization of the perspective-polar transform is
given by
\begin{equation}\label{eq:persp-polar-explicit}
  f^\pp(z, -\xi)
  =\inf\set{\mu>0|\ip z x \le \xi + \mu f(x),\,\forall x}.
\end{equation}
This representation can be obtained by applying the definition of the
gauge polar \eqref{eq:polar-gauge-def} to the perspective transform as
follows:
\begin{align*}
f^\pp(z,-\xi)
  &=\inf\set{\mu>0|\ip z x - \xi\lambda\le\mu f^\pi(x,\lambda),\, \forall
    x, \forall\lambda}
\\&=\inf\set{\mu>0|\ip z x - \xi\lambda\le\mu\lambda
  f(x/\lambda),\,\forall x,\forall\lambda>0}
\\&=\inf\set{\mu>0|\ip{z}{\lambda x} - \xi\lambda\le\mu\lambda
    f(x),\,\forall x,\forall\lambda>0},
\end{align*}
which yields~\eqref{eq:persp-polar-explicit} after dividing through by
$\lambda$. Rockafellar's extension \cite[p.136]{rockafellar} of the
polar gauge transform to nonnegative convex functions that vanish at
the origin coincides with $f^\pp(z,-1)$.

The following theorem provides an alternative characterization of the
perspective-polar transform in terms of the more familiar Fenchel
conjugate $f^\star$. It also provides an expression for the
perspective-polar of $f$ in terms of the Minkowski function generated
by the epigraph of the conjugate of $f$, i.e.,
\[
  \gamma_{\epi f^\star}(x,\tau) := \inf \set{\lambda > 0 \mid (x,\tau) \in \lambda \epi f^\star},
\]
which is a gauge. Nonnegativity of $f$ is not required for the first
part of this result.

\begin{theorem} \label{thm: epi} For any closed proper convex function
  $f$ with $0 \in \dom f$, we have
  $f^{\pi\star}(z, -\xi) = \delta_{\epi f^\star}(z, \xi).$ If, in
  addition, $f$ is nonnegative,
  $f^\pp(z, -\xi) = \gauge{\epi f^\star}(z, \xi)$.
\end{theorem}
\begin{proof}
Because of the assumptions on $f$, we  have $f^\pi (x,0) = \liminf_{\lambda \to 0^+} f^\pi (x, \lambda)$ for each $x \in \Rn$ \cite[Corollary 8.5.2]{rockafellar}. Thus we obtain the following chain of equalities:
\[
\begin{aligned}
  f^{\pi\star}(z, -\xi)
  &= \sup \set{ \langle z,x \rangle - \lambda\xi - f^\pi(x, \lambda) | x \in \Rn,\, \lambda \in \R } \\
  &=\sup \set{ \ip z x  - \lambda\xi - \lambda f(\lambda^{-1}x) |
    x\in\Rn,\ \lambda > 0 }
\\&=\sup \set{ \ip{z}{\lambda y} - \lambda\xi - \lambda f(y) | y \in
  \Rn,\, \lambda > 0 }
\\&=\sup \set{ \lambda \cdot \textstyle\sup_{y} \set{\ip z y - \xi - f(y)}|\lambda>0}
\\&=\sup \set{ \lambda (f^\star(z) - \xi)|\lambda>0}
   =\delta_{\epi f^\star} (z, \xi).
\end{aligned}
\]
This proves the first statement. Now additionally suppose that $f$ is
nonnegative. Because $f^\pi$ is closed, it is identical to its
biconjugate, and so
$f^\pi(x, \lambda) = \delta_{\epi f^\star} ^\star (x, -\lambda)$.
Also, $\epi f^\star$ is closed and convex, and contains the origin
because $f$ is nonnegative.  Therefore, it follows
from \cite[Corollary 15.1.2]{rockafellar} that
\[
  f^\pp(z,-\xi)
  \equiv f^{\pi \circ}(z, -\xi)
  = \delta_{\epi f^\star}^{\star \circ} (z, \xi) = \gauge{\epi f^\star} (z, \xi).
\]
\end{proof}

The following result relates the level sets of the perspective-polar
transform to the level sets of the conjugate perspective. This result is
useful in deriving the constraint sets for certain perspective-dual
problems for which there is no closed form for the perspective polar;
cf.~\cref{sect: logistic}.

\begin{theorem}[Level-set equivalence] \label{thm: starpi}
Let $f:\Rn\to\Rbar_+$ be a 
nonnegative, closed proper convex function with $0 \in \dom f$. 
Then, for any $(z,\xi,\mu)\in\R^n\times\R\times\R$,
\[
  f^\pp(z, \xi) \le \mu
  \quad\iff\quad
  [\ 0\le \mu \ \mbox{ and }\ f^{\star \pi} (z, \mu) \le -\xi\ ].
\]
\end{theorem}
\begin{proof}
The following chain of equivalences follows from \cref{thm: epi}:
\begin{align}
\begin{split}
f^\pp(z, \xi) \le \mu
& \iff \gamma_{\epi f^\star}(z,-\xi) \le \mu \\
 & \iff \inf\set{\lambda > 0| (z,-\xi) \in \lambda \epi f^\star   } \le \mu  \\
 & \iff \inf\set{\lambda > 0| f^\star(z/\lambda)\le -\xi/ \lambda } \le \mu \\
 & \iff \inf\set{\lambda > 0| f^{\star \pi}(z, \lambda)\le -\xi   } \le \mu.
 \label{eq: ppinf}
 \end{split}
 \end{align}
Define $\alpha = \inf\set{\lambda > 0| f^{\star \pi}(z, \lambda)\le -\xi}$.

We first show that $f^\pp(z, \xi) \le \mu$ implies $0\le \mu$ and  $f^{\star \pi} (z, \mu) \le -\xi$.
By \eqref{eq: ppinf}, $0\le\alpha\le \mu$. 
If
$\alpha <\mu$, there exists $\lambda$ with $0 < \lambda < \mu$ such
that $f^{\star \pi}(z, \lambda)\le -\xi.$ Because $f$ is nonnegative,
$\mu f \ge \lambda f$, and thus $(\mu f)^\conj \le (\lambda f)^\conj$.
In particular,
\[
  f^{\star \pi}(z, \mu) = (\mu f)^\conj (z) \le (\lambda f)^\conj(z) =
  f^{\star \pi}(z, \lambda)\le -\xi.
\]
On the other hand, if $\alpha = \mu$, there exists a sequence
$\lambda_k \to \mu$ such that $ f^{\star \pi}(z, \lambda_k) \le - \xi$
for each $k$.  Now by the lower semi-continuity of $f^{\star \pi}$, we
obtain
\begin{equation*}\label{eq:11}
  f^{\star \pi}(z, \mu) \le \liminf_{k \to \infty} f^{\star \pi}(z, \lambda_k) \le - \xi.
\end{equation*}
This establishes the forward implication of the theorem.

For the reverse implication, suppose $0 \le \mu$ and
$f^{\star \pi} (z, \mu) \le -\xi$.  If $0 < \mu$, it follows
from~\eqref{eq: ppinf} that $f^\pp(z, \xi) \le \mu$. Now suppose
otherwise that $\mu=0.$ We want to show $f^\pp(z, \xi) \le 0$.  By
hypothesis,
$(z,0,-\xi) \in \epi f^{\star \pi}$. Thus there
exists a sequence $(z_k, \mu_k, r_k)$ with
$\lim_{k \to \infty} (z_k, \mu_k, r_k) = (z, 0, -\xi)$ and
$f^{\star\pi}(z_k, \mu_k) \le r_k$ for all $k$.  With no loss in
generality, we can assume that $\mu_k >0$ for all $k$. Then for each
$k$, we have $\mu_k f^\star(z_k/\mu_k) \le r_k$ for which we have the
following equivalences:
\begin{align*}
\mu_k f^\star(z_k/\mu_k) \le r_k 
& \iff \sup_w \set{ \ip{w}{z_k} - \mu_k f(w) } \le r_k \\
& \iff \ip{w}{z_k}  \le  r_k+ \mu_k f(w), \, \forall w \in \Rn \\
& \iff \mu_k \ge \inf \set{ \lambda >0 |\ip{w}{z_k}  \le  r_k+ \lambda f(w), \, \forall w \in \Rn } \\
&\underset{\eqref{eq:persp-polar-explicit}}{\iff} \mu_k \ge f^\pp(z_k,-r_k),
\end{align*}
which gives $f^\pp(z,\xi)\le 0=\mu$ in the limit, since $f^\pp$ is closed.
\end{proof}

\subsubsection{Calculus rules} \label{sec:perspective-polar-examples}
Two useful calculus rules are now developed that govern the
perspective-polar transform when applied to gauge functions and
separable sums.

\begin{example}[Gauge functions]
  \label{ex: gauge} Suppose that $f$ is a closed proper gauge. Then
  \begin{equation*}
    f^\pp (z, \xi) = f^\circ(z) + \delta_{\R_-} (\xi).
  \end{equation*}
  Use expression~\eqref{eq:persp-polar-explicit} for this
  derivation. When $\xi >0$, take $x=0$ in the infimum
  in~\eqref{eq:persp-polar-explicit} to deduce that
  $f^\pp (z, \xi) = + \infty.$ On the other hand, when $\xi \le 0$,
  the positive homogeneity of $f$ implies that
  $f^\pp(z, \xi) = f^\circ(z)$. We leave the details to the reader.
  More generally, if $f$ vanishes at the origin, then
  $f^\pp (z, \xi) = + \infty$ for all $\xi >0.$

\end{example}

\begin{example}[Separable sums] \label{ex: sums} %
  Suppose that $f(x) := \sum_{i=1}^n f_i(x_i),$ where each convex
  function $f_i: \R^{n_i} \to \Rbar_+$ is nonnegative. Then a
  straightforward computation shows that
  $f^\pi (x, \lambda) = \sum_{i=1}^n f_i^\pi(x_i, \lambda)$. Furthermore,
  taking into account \cite[Proposition 2.4]{gaugepaper}, which
  expresses the polar of a separable sum of gauges, we deduce
  \[
    f^\pp(z, \xi) = \max_{i=1, \ldots, n}~ f_i^\pp(z_i, \xi).
  \]
\end{example}

\subsection{Derivation of the perspective dual via lifting} \label{sect: derive}

We now derive the relationship between the primal and dual problems
\eqref{eqn:target} and \eqref{eqn:persp_dual} by lifting
\eqref{eqn:target} to an equivalent gauge optimization problem, and
then recognizing \eqref{eqn:persp_dual} as its gauge dual.

\begin{theorem}[Gauge lifting of the primal] \label{thm: reform} A
  point $x^*$ is optimal for~\eqref{eqn:target} if and only if
  $(x^*, 1)$ is optimal for the gauge problem
  \begin{equation} \label{eqn:prim_equiv}
    \minimize{x, \lambda} \quad f^{\pi}(x, \lambda)
    \quad\st\quad
    \rho
    \left(\begin{bmatrix} b \\ 1 \\ 1 \end{bmatrix}
      -\begin{bmatrix}
        A & 0\\
        0 & 1\\
        0 &  0
      \end{bmatrix}
      \begin{bmatrix}x \\ \lambda\end{bmatrix}
    \right)\leq \sigma,
\end{equation}
where
$\rho(z, \mu, \tau) := g^\pi(z,\tau) + \delta_{\{0\}}(\mu)$ is a gauge function.
\end{theorem}

\begin{proof}
  By definition of $f^{\pi}$, $x^*$ is optimal for~\eqref{eqn:target}
  if and only if the pair $(x^*, 1)$ is optimal for
  \[
    \minimize{x,\lambda} \quad f^\pi(x,\ \lambda)
    \quad\st\quad \lambda=1,\ g^{\pi}(b-Ax,1)\leq \sigma.
  \]
The following equivalence follows from the definition of $\rho$:
\[
   [\,  \lambda=1 \quad \mbox{and} \quad g^\pi (b-Ax, 1) \le \sigma\, ]
   \quad
   \iff\quad  \rho (b-Ax, 1-\lambda, 1)  \le \sigma.
\]
Thus we arrive at the constraint expressed in~\eqref{eqn:prim_equiv}.
\end{proof}

\begin{corollary}[Gauge dual]\label{cor:gauge_dual_lift}
  Problem~\eqref{eqn:persp_dual} is the gauge dual of \eqref{eqn:prim_equiv}.
\end{corollary}
\begin{proof}
  It follows from the canonical dual pairing~\eqref{eq:gauge-primal}
  and~\eqref{eq:gauge-dual} that the gauge dual
  of~\eqref{eqn:prim_equiv} is
  \begin{equation} \label{eqn:persp_dual_lift}
    \begin{array}{ll}
      \minimize{y,\alpha, \mu} &
           f^{\pi \circ}
                \left(
                  \begin{bmatrix}
                    A^T & 0 & 0\\
                    0 & 1 & 0
                  \end{bmatrix}
                  \begin{bmatrix}
                    y\\ \alpha\\ \mu
                  \end{bmatrix}
                \right)
  \\[15pt]\st &
      \ip{(y,\alpha,\mu)}{(b,1,1)} - \sigma \rho^{\circ}(y, \alpha, \mu) \ge 1.
    \end{array}
\end{equation}
Because $\rho$ is separable in $(z,\mu)$ and $\beta$, it follows from
\cite[Proposition 2.4]{gaugepaper} that
\[
  \rho^\circ (y, \alpha, \mu)
  = \max \set{g^{\pi \circ} (y, \mu) ,\ \delta_{\{0\}}^\circ(\alpha)}.
\]
Since $\delta_{\{0\}}^\circ(\alpha)$ is identically zero, the result follows.
\end{proof}

The next result generalizes the gauge duality result of
\cref{thm:gauge-duality} to the case where $f$ and $g$ are convex and nonnegative
but not necessarily gauges. We parallel the construction
in~\eqref{gauge feasible sets}, and for this section only redefine the
feasible sets by
\begin{align*}
  \Fscr_p &:=\set{u|g(b-u) \le \sigma}
\\\Fscr_d &:=\set{(y,\alpha,\mu)| \ip b y -\sigma g^\pp(y,\mu)\geq 1-(\alpha + \mu) }.
\end{align*}
Thus, \eqref{eqn:target} is \emph{relatively strictly feasible} if
\[
  A^{-1} \ri \Fscr_p \cap (\ri \dom f )\ne \emptyset.
\]
Similarly, \eqref{eqn:persp_dual} is \emph{relatively strictly
  feasible} if there exists a triple $(y,\alpha,\mu)$ such that
\[
  (A\T y,\alpha)\in\ri \dom f^\pp \quad \text{and} \quad (y,\alpha,\mu) \in \ri \Fscr_d.
\] {\em Strict feasibility} follows the same definitions, where the
operation $\ri$ is replaced by $\inter$.  

\begin{theorem}[Perspective duality]
  \label{thm:per_duality}
  Let $\nu_p$ and $\nu_d$, respectively, denote the optimal values of
  the pair \eqref{eqn:target} and \eqref{eqn:persp_dual}.  Then the
  following relationships hold for the perspective dual pair
  \eqref{eqn:target} and \eqref{eqn:persp_dual}.
\begin{enumerate}[\rm{(a)}]
\item (Basic Inequalities) It is always the case that
\[
\mbox{(\textit{i}) $(1/\nu_p)\leq \nu_d\ $  and \ (\textit{ii}) $(1/\nu_d)\le \nu_p$.}
\]
Thus, $\nu_p=0$ and $\nu_d=0$, respectively, imply that
\eqref{eqn:persp_dual} and \eqref{eqn:target} are infeasible.

\item (Weak duality) If $x$ and $(y,\alpha,\mu)$ are primal and dual feasible, then
  \begin{equation*}
    1\leq \nu_p\nu_d \le  f(x)\cdot f^\pp(A\T y,\alpha).
  \end{equation*}
\item (Strong duality) If the dual (resp. primal) is feasible and the
  primal (resp. dual) is relatively strictly feasible, then
  $\nu_p\nu_d = 1$ and the perspective dual (resp. primal) attains its
  optimal value.
\end{enumerate}
\end{theorem}
\begin{proof}
  Parts (a) and (b) follow immediately from the analogous result in
  \cref{thm:gauge-duality}, together with \cref{thm: reform} and
  \cref{cor:gauge_dual_lift}.

  Next we demonstrate that~\eqref{eqn:target} is relatively strictly
  feasible if and only if~\eqref{eqn:prim_equiv} is relatively
  strictly feasible. 
  By the description of relative interiors of sublevel sets given in 
  \cite[Theorem 7.6]{rockafellar}, ~\eqref{eqn:prim_equiv} is
  relatively strictly feasible if and only if there exists a point
  $(x,1) \in \ri \dom f^\pi$ such that
  \[
    (b-Ax,  0,1) \in \ri \dom \rho \quad \text{and} \quad \rho(b-Ax,0,1 ) = g(b-Ax) <\sigma.
  \]
We now seek a description of $ \ri \dom f^\pi$. We have
\begin{align*}
  \dom f^\pi &= \set{ (x,\mu) | f^\pi(x,\mu) <\infty} \\
             &= \cl \left( \{0\} \cup \set{ (x,\mu) | \mu>0, \, f(x/\mu) <\infty} \right)
               = \cl \cone{ (\dom f \times \{1\} )}.
\end{align*}
By \cite[Corollary 6.8.1]{rockafellar}, the above description yields
\[
  \ri\dom f^\pi = \set{ (x,\mu) | \mu >0, \, x \in \mu \, \ri\dom f }.
\]
Thus $(x,1) \in \ri \dom f^\pi $ if and only if $x \in \ri\dom
f$. Similarly,
\[
  \dom \rho = \set{(y,0,\mu) | (y,\mu) \in \dom g^\pi },
\]
and so
\[
  \ri\dom\rho
  = \set{ (y,0,\mu) | (y,\mu)\in  \ri\dom g^\pi}
  = \set{ (y,0,\mu) \mid \mu>0, y \in \mu\,  \ri\dom g}.
\]
In particular, the condition $(b-Ax, 0, 1) \in \ri\dom\rho$ is
equivalent to $b-Ax \in \ri\dom g$. Thus the conditions for relative
strict feasibility of~\eqref{eqn:prim_equiv} and~\eqref{eqn:target}
are identical.

A similar argument verifies that \eqref{eqn:persp_dual} is
relatively strictly feasible if and only if \eqref{eqn:persp_dual_lift} is
relatively strictly feasible.  Strong duality then follows from
relative interiority, \cref{cor:gauge_dual_lift}, \cref{thm: reform},
and the analogous strong-duality result in
\cref{thm:gauge-duality}.
\end{proof}

\subsection{Optimality conditions}

The following result generalizes \cref{thm:gauge_opt_pert} to include
the perspective-dual pair.

\begin{theorem}[Perspective optimality] \label{thm:peropt}
  Suppose \eqref{eqn:target} is strictly feasible. Then the tuple
  $(x^*,y^*,\alpha^*,\mu^*)$ is perspective primal-dual optimal if and only
  if
    \begin{align*}
      g(b-Ax^*) &= \sigma \quad \text{or}  \quad g^\pp(y^*,\mu^*) = 0 &&\hbox{(primal activity)}
    \\\ip b {y^*} -\sigma g^\pp(y^*, \mu^*) &= 1 - (\alpha^*+\mu^*)
                         &&\hbox{(dual activity)}
    \\\ip{x^*}{A\T y^*} + \alpha^* &= f(x^*)\cdot f^\pp (A\T y^*,\alpha^*)
                         &&\hbox{(objective alignment)}
    \\\ip{b - Ax^*}{y^*} + \mu^* &= g(b-Ax^*)\cdot g^\pp(y^*,\mu^*).
                         &&\hbox{(constraint alignment)}
    \end{align*}
\end{theorem}
\begin{proof}
  By construction, $x^*$ is optimal for \eqref{eqn:target} if and only
  if $(x^*,1)$ is optimal for its gauge
  reformulation~\eqref{eqn:prim_equiv}. Apply
  \cref{thm:gauge_opt_pert} to~\eqref{eqn:prim_equiv} and the
  corresponding gauge dual~\eqref{eqn:persp_dual} to obtain the
  required conditions.
\end{proof}

\noindent The following result mirrors \cref{cor:recovery} for the
perspective-duality case.

\begin{corollary}[Perspective primal-dual recovery] \label{thm:
    perrecovery} Suppose that the primal \eqref{eqn:target} is
  strictly feasible. If $(y^*, \alpha^*, \mu^*)$ is optimal for
  \eqref{eqn:persp_dual}, then for any primal feasible $x \in \Rn$, the following
  conditions are equivalent:
  \begin{enumerate}[(a)]
  \item $x$ is optimal for \eqref{eqn:target};
  \item
    $\ip{x}{A\T y^*} + \alpha^* =
    f(x)\cdot f^\pp(A\T y^*,  \alpha^*)$ and
    $(b-Ax, 1) \in \sigma\partial g^\pp(y^*, \mu^*);$
  \item $A\T y^* \in f^\pp(A\T y^*,\alpha^*)\cdot \partial f(x)$
    and $(b-Ax, 1) \in \sigma\partial g^\pp(y^*, \mu^*).$
  \end{enumerate}
\end{corollary}
\begin{proof}
  By construction, $x$ is optimal for \eqref{eqn:target} if and only
  if $(x,1)$ is optimal for its gauge
  reformulation~\eqref{eqn:prim_equiv}. Apply
  \cref{cor:recovery} to \eqref{eqn:prim_equiv} and its gauge
  dual \eqref{eqn:persp_dual} to obtain the equivalence of (a) and
  (b). To show the equivalence of (b) and (c), note that by the
  polar-gauge inequality,
  $\ip{(x,1)}{(A\T y^*, \alpha^*)} \le f^\pi(x,1) \cdot f^\pp(A\T y^*,
  \alpha^*)$ for all $x$, or equivalently,
  \begin{equation*}
    \ip{x}{A\T y^*} + \alpha^*
    \le f(x) \cdot f^\pp(A\T y^*, \alpha^*), \quad\forall x.
\end{equation*}
The inequality is tight for a fixed $x$ if and only if $x$ minimizes
the function \\
$h := f^\pp(A\T y^*, \alpha^*)f(\cdot) - \ip{\cdot}{A\T y^*} -
\alpha^*.$ This in turn is equivalent to $0 \in \partial h(x)$,
or
\[ A\T y^* \in f^\pp(A\T y^*, \alpha^*) \cdot \partial f(x). \] This
shows the equivalence of (b) and (c) and completes the proof.
\end{proof}

\Cref{sect: recovery_ex} illustrates an application of \cref{thm: perrecovery}
for recovering primal optimal solutions from perspective-dual optimal
solutions.

\subsection{Reformulations of the perspective dual}
Two reformulations of the perspective dual \eqref{eqn:persp_dual} may be useful depending on the functions $f$ and $g$ involved in \eqref{eqn:target}.
First, an important simplification of the perspective dual occurs when one or both of these functions are gauges.

\begin{corollary}[Simplification for gauges]
  \label{cor: gaugeobj}
  If $f$ is a gauge, then a triple $(y^*, \alpha^*, \mu^*)$ is optimal
  for~\eqref{eqn:persp_dual} if and only if $\alpha^*\le 0$ and $(y^*, \mu^*)$ is optimal
  for
  \[
    \minimize{y,\alpha} \enspace f^\circ \big( A\T y \big)
    \enspace\st\enspace
    \ip{b}{y} - \sigma g^\pp (y, \mu) \ge 1 - \mu.
  \]
  If, in addition, $g$ is a gauge, then a triple $(y^*, \alpha^*, \mu^*)$ is
  optimal for \eqref{eqn:persp_dual} if and only if $\alpha^*\le 0$, $\mu^*\le 0$, and $y^*$ solves \eqref{eq:gauge-dual}.
\end{corollary}
\begin{proof}
  Follows from the formulas for $f^\pp$ and $g^\pp$ established
  in \cref{sec:perspective-polar-examples}.
\end{proof}

\cref{thm: starpi} also allows us to express the level sets of $g^\pp$
in terms of its conjugate polar as in the following corollary.

\begin{corollary}
  \label{cor: logper}
  The point $(y^*, \alpha^*, \mu^*)$ is optimal for \eqref{eqn:persp_dual}
  if and only if there exists a scalar $\xi^*$ such that
  $(y^*, \alpha^*, \mu^*, \xi^*)$ is optimal for the problem 
\begin{equation*}
  \begin{array}{ll}
    \minimize{y,\alpha,\mu,\xi} & f^\pp( A\T y, \alpha )
 \\ \st       & \begin{aligned}[t]
                  \ip b y & - \sigma \xi =  1 - (\alpha +\mu)
                , \quad g^{\star \pi} (y, \xi) \le -\mu,\  \xi \ge 0.
                \end{aligned}
  \end{array}
\end{equation*}
\end{corollary}
\begin{proof}
By introducing the variable $\xi := (\ip b y + \alpha +\mu -1)/\sigma$ in
  \eqref{eqn:persp_dual}, the result follows from
  \cref{thm: starpi}.
\end{proof}

\section{Examples: piecewise linear-quadratic and GLM constraints}
\label{sect: casestudies}

From a computational standpoint, the perspective-dual formulation may
be an attractive alternative to the original primal problem. The
efficiency of this approach requires that the dual constraints are in
some sense more tractable than those of the primal. For example, we
may consider the dual feasible set ``easy'' if it admits an efficient
procedure for projecting onto that set.  In this section, we examine
two special cases that admit tractable dual problems in this
sense. The first case is the family of piecewise linear quadratic
(PLQ) functions, introduced by Rockafellar~ \cite{Rock:88} and
subsequently examined by Rockafellar and
Wets~\cite[p.440]{rockafellarwets}, and Aravkin, Burke, and
Pillonetto~\cite{stablespline}. The second case is when $g$ is a
Bregman divergence arising from a maximum likelihood estimation
problem over a family of exponentially distributed random variables.

For this section only, we will assume for the sake of simplicity
that the objective $f$ is a gauge, so that the perspective dual in
each of this cases simplifies as in \cref{cor: gaugeobj}. The more
general case still applies.

\subsection{PLQ constraints} \label{sect: plq}

The family of PLQ functions is a large class of convex functions that
includes such commonly used penalties as the Huber function, the
Vapnik $\epsilon$-loss, and the hinge loss. The last two are used in
support-vector regression and classification~\cite{stablespline}. PLQ
functions take the form
\begin{equation}  \label{eq: plqdef}
 g(y) = \sup_{u \in U}
 \set{ \ip{u}{By+b} - \tfrac12 \|Lu\|_2^2 },
\quad \Uscr := \set{ u \in \R^\ell | Wu \le w },
\end{equation}
where $g$ is defined by linear operators $L\in\R^{\ell \times \ell}$ and
$W\in\R^{k \times \ell}$, a vector $w \in \R^k$, and an injective affine
transformation $B(\cdot) +b$ from $\R^k$ to $\R^\ell$. We may assume
without loss of generality that $B(\cdot)+b$ is the identity
transformation, since the primal problem \eqref{eqn:target} already
allows for composition of the constraint function $g$ with an affine
transformation. We also assume that $\Uscr$ contains the origin, which
implies that $g$ is nonnegative and thus can be interpreted as a
penalty function. Aravkin, Burke, and Pillonetto~\cite{stablespline}
describe a range of PLQ functions that often appear in applications.

The conjugate representation of  $g$, given by
\begin{equation*}
 g^\star(y) = \delta_\Uscr(y) + \tfrac{1}{2}\|L y \|^2, \label{eq: plq}
\end{equation*}
is useful for deriving its polar perspective $g^\pp$.  In the
following discussion, it is convenient to interpret the quadratic
function $-(1/2\mu)\|Ly\|^2$ as a closed convex function of $\mu\in\R_-$, and
thus when $\mu = 0$, we make the definition
$-(1/2\mu)\|Ly\|^2 = \delta_{\{0\}}(y)$.

\begin{theorem}
  \label{thm: plqpipolar}
  If $g$ is a PLQ function, then
\begin{align*}
  g^\pp(y, \mu)
  &= \delta_{\R_-}(\mu) +
    \max\left\{ \gauge{\Uscr}(y),\, - (1/2\mu)\|Ly\|^2 \right\}
\\&= \delta_{\R_-}(\mu) + \max\left\{  - (1/2\mu) \|Ly\|^2,
    \, \max_{i = 1, \ldots, k} \set{ W_i\T y / w_i } \right\},
\end{align*}
where $W_1^T, \ldots , W_k^T$ are the rows of $W$ that define $\Uscr$ in~\eqref{eq: plqdef}. 
\end{theorem}
\begin{proof}
  First observe that when $g$ is PLQ,
  $\epi g^\star=\set{(y,\tau) | y \in \Uscr,\, \frac{1}{2}\|Ly\|^2 \leq \tau}$.
  Apply \cref{thm: epi} and simplify to obtain the chain of equalities
\begin{align*}
  g^\pp(y, \mu)
   = \gauge{\epi g^\star}(y, - \mu)
  &=\inf\set{\lambda > 0 | (y, - \mu)\in \lambda \epi g^\star}
\\&=\inf\set{\lambda > 0 | y/\lambda \in \Uscr, \, (1/2\lambda^2)\|Ly\|^2  \le - \mu/\lambda  }
\\&=\delta_{\R_-}(\mu) + \max\left\{ \gauge{\Uscr}(y),\, - (1/2\mu)\|Ly\|^2 \right\}.
\end{align*}
Because $\Uscr$ is polyhedral, we can make the explicit description
\begin{align*}
 \gauge{\Uscr}(y)
 &= \inf \set{ \lambda > 0 | y \in \lambda \Uscr } \\
 &= \inf \set{ \lambda > 0 | W(y/\lambda) \le w }
  = \max \left\{ 0,  \, \max_{i = 1, \ldots, k} \{ W_i\T y / w_i \} \right\}.
\end{align*}
This follows from considering cases on the signs of the $W_i\T y$, and noting that $w \ge0$ 
because $\Uscr$ contains the origin. Combining the above results, the theorem is proved.
\end{proof}

The next example illustrates how \cref{thm: plqpipolar} can be applied
to compute the perspective-polar transform of the Huber function.

\begin{example}[Huber function]
  \label{ex: huber}
  The Huber function~\cite{Hub}, 
  which is a smooth approximation to the absolute value function, is
  also its Moreau envelope of order $\eta$. Thus it can be stated in
  conjugate form as
  \[
    h_\eta(x)
    = \sup_{u\in[-\eta,\eta]}\Set{ux-(\eta/2)u^2}
    = \sup_{u}\Set{ux-[\delta_{[-\eta,\eta]}(u)+(\eta/2)u^2]},
  \]
  which reveals $h_\eta^\star(y) = \delta_{[-\eta,\eta]}(y)+(\eta/2) y^2$.
  We then apply \cref{thm: epi} to obtain
  \begin{align*}
    h_\eta^\pp(z, \xi)
    &= \gauge{h^\star_\eta}(z,-\xi)
  \\&= \inf\set{\lambda>0|(z,-\xi)\in\lambda\epi h_\eta^\star}
  \\&= \inf\set{\lambda>0||z|/\lambda \le \eta,\ (\eta/2\lambda)z^2\le-\xi}
  \\&= \delta_{\R_-}(\xi) + \max \left\{ |z|/\eta,\ -(\eta/2\xi) z^2 \right\}.
  \end{align*}
\end{example}
Note that this can easily be extended beyond the univariate case to a separable sum by applying the result of Example~\ref{ex: sums}.

We can now write down an explicit formulation of the perspective dual problem~\eqref{eqn:persp_dual} when the primal problem~\eqref{eqn:target} has a
PLQ-constrained feasible region (i.e.,
$g$ is PLQ) and a gauge objective (i.e., $f$ is a closed gauge). The constraint set of~\eqref{eqn:persp_dual} simplifies significantly so
that, for example, a first-order projection method might be applied to
solve the problem. Apply \cref{thm: plqpipolar} and introduce a
scalar variable
$\xi$ to rephrase the dual problem~\eqref{eqn:persp_dual} as
\begin{equation} \label{eq: fpolar_hub}
  \begin{array}{ll}
     \minimize{y,\,\mu,\,\xi} & f^{\circ} \big( A\T y \big)
    \\\st &\begin{aligned}[t]
          \ip b y &+ \mu - \sigma \xi = 1, \ \mu \le 0,\ \xi \ge 0,
         \\       Wy &\le \xi w,\ -(1/2\mu)\|Ly\|^2 \le \xi.
          \end{aligned}
  \end{array}
\end{equation}
We can further simplify the constraint set using the fact that
\begin{equation}\label{eq:projsub}
  \big[\ \|Ly\|^2 \le - 2\mu \xi \text{ and }  \mu \le 0, \, \xi \ge0\ \big]
  \iff
  \left\|
    \begin{bmatrix}
      2Ly  \\ \xi+ 2\mu
    \end{bmatrix}
  \right\|_2 \le \xi -  2 \mu,
\end{equation}
Thus, projecting a point $\ybar$ onto the feasible set of \eqref{eq: fpolar_hub}
is equivalent to solving a second-order cone program (SOCP).
In many important cases, the operator
$L$ is extremely sparse. For example, when
$g$ is a sum of separable Huber functions, we have $L=\sqrt\eta
I$.  Hence in many practical cases, particularly when $m \ll n$ and the dual variables are low-dimensional, this projection problem could be
solved efficiently using SOCP solvers that take advantage of sparsity,
e.g., Gurobi~\cite{gurobi}.

\subsection{Generalized linear models and the Bregman divergence}
\label{sec:glms}

Suppose we are given a data set
$\{(a_i,b_i)\}_{i=1}^m\subseteq\R^{n+1}$, where each vector $a_i$
describes features associated with observations $b_i$.  Assume that
the vector $b$ of observations is distributed according to an
exponential density
$ p(y\mid\theta) = \exp[ \ip\theta y - \phi^\star(\theta) - p_0(y) ],
$ where the conjugate of $\phi:\Rn\to\R$ is the cumulant generating
function of the distribution and $p_0:\Rn\to\R$ serves to normalize
the distribution. We assume that $\phi$ is a closed convex function of
the Legendre type \cite[p.258]{rockafellar}. The maximum likelihood
estimate (MLE) can be obtained as the maximizer of the log-likelihood
function $\log p(y\mid \theta)$.

In applications that impose an \textit{a priori} distribution on the parameters, the goal is to find an approximation to the MLE estimate that penalizes a regularization function $f$ (a surrogate for the prior).
  We assume a linear dependence between the parameters and
feature vectors, and thus set $\theta=Ax$, where the matrix $A$ has
rows $a_i$. A regularized MLE estimate could be obtained by solving
the constrained problem
\[
\minimize{x}\enspace f(x) \enspace\st\enspace d_{\phi^\conj}(Ax;\nabla\phi(b))\le\sigma,
\]
where $d_\phi(v;w) := \phi(v) - \phi(w) - \ip{\nabla\phi(w)}{v-w}$ is
the {\em Bregman divergence} function, and $\sigma$ is a positive
parameter that controls the divergence between the linear model $Ax$
and the first-moment $\nabla\phi(b)$ relative to the density
defined by $\phi$~\cite{banerjee2005clustering}.

We use \cref{cor: logper} to derive the perspective dual, which requires the computation of the conjugate of $g(z):=d_{\phi^\conj}(z;\nabla\phi(b))$:
\begin{align*}
  g^\star(y)
  &= \sup_z\set{\ip z y - d_{\phi^\star}(z;\nabla\phi(b))}
\\&= \sup_z\set{\ip z y - \phi^\star(z)+\phi^\conj(\nabla\phi(b)) +
      \ip{b}{z-\nabla\phi(b)}}
\\&= \phi^\star(\nabla\phi(b))-\ip{b}{\nabla\phi(b)} + \phi(y+b),
\end{align*}
where we simplify the expression using the inverse relationship
between the gradients of $\phi$ and its conjugate. Assume for simplicity that $f$ is a gauge, which is typical when it serves as a regularization function. In that case, the perspective dual reduces to 
\begin{equation} \label{eqn:bregman perdual}
  \begin{array}{ll}
    \minimize{y,\mu,\xi} & f^\circ(A\T y)
    \\\st & \phi^\pi(y+\xi b,\xi)
              \le \xi [\ip{b}{\nabla\phi(b)}-\phi^\star(\nabla\phi(b))-\sigma]-1,\ \xi\ge0;
  \end{array}
\end{equation} cf.~Corollaries~\ref{cor: gaugeobj} and~\ref{cor: logper}.

\begin{example}[Gaussian distribution]
As a first example, consider the case where the $b_i$ are distributed as independent Gaussian variables with unit variance. In this case, $\phi:=\half\|\cdot\|^2$ and the above constraints specialize to
\[ \tfrac1{2\xi}\|y\|^2 + \ip b y\le-(1+\sigma\xi),\ \xi\ge0.
\] This is an example of a PLQ constraint, which falls into the category of problems described in \cref{sect: plq}.
\end{example}

\begin{example}[Poisson distribution] \label{sect: logistic} Consider
  the case where the observations $b_i$ are independent Poisson
  observations, which corresponds to
  $\phi(\theta)=\theta\log\theta-\theta$ and $\phi^\star(y)=e^y$.
  Straightforward calculations show that the perspective dual
  constraints for the Poisson case reduce to
  \[
    \sum_{i=1}^m z_i \log (z_i/\xi) \le \beta \xi + \sum_{i=1}^m z_i -
    (1+\sigma\xi), \quad z= y+ \xi b,\quad \xi\ge0,
  \]
  where $\beta = \sum_{i=1}^m( b_i + b_i \log b_i)$ is a constant. By
  introducing new variables, this can be further simplified to require
  only affine constraints and $m$ relative-entropy constraints.  To
  solve projection subproblems onto a constraint set of this form, we
  note that
\[
  F(x,y,r) = 400 ({-}{\log(x/y)} -\log(\log(x/y) - r/y) -4\log(y))
\]
is a self-concordant barrier %
for the set
$
  \set{ (x,y,r) | y >0, \ y\log(y/x) \le r },
$
which is the epigraph of the relative entropy function; see Nesterov and Nemirovski~\cite[Proposition 5.1.4]{ipmethods} and Boyd and Vandenberghe~\cite[Example 9.8]{boydvan}.  Standard interior methods can therefore be used to project onto the constraint set.
\end{example}

\begin{example}[Bernoulli distribution]
  When the observations $b_i$ are independent Bernoulli observations,
  which corresponds to
  $\phi(\theta)=\theta\log\theta + (1-\theta)\log(1-\theta)$ and
  $\phi^\star(y)=\log(1+e^y)$, the perspective dual constraints
  in~\eqref{eqn:bregman perdual} reduce to
\[
  \sum_{i=1}^m \left[ z_i \log (z_i/\xi) + (\xi - z_i )\log ((\xi - z_i)/\xi)\right] \le \beta \xi - (1+\sigma\xi),
  \quad z= y+ \xi b,  \quad  \xi\ge0,
\]
where $\beta = \sum_{i=1}^m (b_i \log b_i + (1-b_i) \log (1-b_i))$ is a constant. By introducing new variables, this can be rewritten with only affine constraints and $2m$ relative-entropy constraints. Thus the projection subproblems can be solved as in the Poisson case.
\end{example}

\section{Examples: recovering primal solutions} \label{sect: recovery_ex}

Once we have solved the gauge or perspective dual problems, we have
two available approaches for recovering a corresponding primal optimal
solution. If we applied a (Lagrange) primal-dual algorithm (e.g., the
algorithm of Chambolle and Pock~\cite{cp}) to solve the dual, then
\cref{thm:lagrange} gives a direct recipe for constructing a primal
solution from the algorithm's output. On the other hand, if we applied
a primal-only algorithm to solve the dual, we must instead rely on
\cref{cor:recovery} or \cref{thm: perrecovery} to recover a primal
solution. Interestingly, the alignment conditions in these theorems
can provide insight into the structure of the primal optimal solution,
as illustrated by the following examples.

\subsection{Recovery for basis pursuit
  denoising} \label{example1}

Our first example illustrates how \cref{cor:recovery}
can be used to recover primal optimal solutions from dual optimal
solutions for a simple gauge problem. Consider the gauge dual pair
\begin{subequations}
  \begin{alignat}{4}
    &\minimize{x} &\enspace &\|x\|_1 &\enspace&\st &\enspace  \|b-Ax\|_2 &\le \sigma \label{eq:rec_ex}
  \\&\minimize{y} &\enspace &\|A\T y\|_\infty &\enspace&\st&\enspace\ip b y - \sigma \|y\|_2 &\ge 1, \label{eq:rec_ex_dual}
  \end{alignat}
\end{subequations}
which corresponds to the basis pursuit denoising problem.  The 1-norm
in the primal objective encourages sparsity in $x$, while the
constraint enforces a maximum deviation between a forward model $Ax$
and observations $b$.

Let $y^*$ be optimal for the dual problem \eqref{eq:rec_ex_dual}, and set
$z = A\T y^*$. Define the active set
\[
  I(z) = \set{ i | |z_i| = \|z\|_\infty}
\]
as the set of indices of $z$ that achieve the optimal objective value
of the gauge dual. We use \cref{cor:recovery} to determine properties
of a primal solution $x^*$. In particular, the first part of
\cref{cor:recovery}(b) holds if and only if $x_i^*=0$ for all
$i \notin I(z)$, and $\sign (x_i^*) = \sign (z_i)$ for all
$i \in I(z).$ Thus, the maximal-in-modulus elements of $A\T y^*$
determine the support for any primal optimal solution $x^*$.  The
second condition in \cref{cor:recovery}(b) holds if and only if
$b-Ax = \sigma y^*/\|y^*\|_2$.  In order to satisfy this last
condition, we solve the least-squares problem restricted to the
support of the solution:
\[
  \minimize{x}\quad\left\|b - Ax - \sigma \left(y^*/\|y^*\|_2\right)\right\|_2^2
  \quad\st\quad
  x_i = 0\ \ \forall i\notin I(z).
\]
(Note that $y^*\ne0$, otherwise the primal problem is
infeasible.)  The efficiency of this least-squares solve depends on
the number of elements in $I(z)$. For many applications of basis
pursuit denoising, for example, we expect the support to be small
relative to the length of $x$, and in that case, the least-squares
recovery problem is expected to be a relatively inexpensive
subproblem.  We may interpret the role of the dual problem as that of
determining the optimal support of the primal, and the role of the
above least-squares problem as recovering the actual values of the
support.

\subsection{Sparse recovery with Huber misfit} \label{ex: one_huber}

For an example where the constraint is not a gauge function, consider the variant of \eqref{eq:rec_ex}
\begin{equation} \label{eq:sparse-huber}
 \minimize{x} \enspace \|x\|_1
 \enspace \st\enspace h(b-Ax) \le\sigma,
 \textt{with}
 h(r)=\sum_{i=1}^m h_{\eta}(r_i),
\end{equation}
where $h_\eta$ is the Huber function; cf.~\cref{ex: huber}. This problem
corresponds to~\eqref{eqn:target} with $f(x)=\|x\|_1$ and
$g=h$. Suppose that the tuple $(y, \alpha, \mu)$, with $\mu<0$, is
optimal for the perspective dual, and that~\eqref{eqn:target} attains
its optimal value. Because $f$ is a gauge, \cref{cor: gaugeobj}
asserts that $\alpha =0$, and thus \cref{thm: perrecovery}(b) reduces
to the conditions
\begin{subequations}
\begin{align}
  \ip{x}{A\T y} &= f(x)\cdot f^{\circ} (A\T y)   \label{ex: align}
\\(b-Ax, 1)     &\in \sigma\partial h^\pp(y, \mu).
\end{align}
\end{subequations}
As we did for the related example in \cref{example1}, we
use~\eqref{ex: align} to deduce the support of the optimal primal
solution. It follows from \cref{thm: plqpipolar} that because $g$ is
PLQ,
\[
  h^\pp(y, \mu) =
  \delta_{\R_-}(\mu)
  + \max\left(\max_{i =
      1, \ldots, k} \{ W_i\T y/ w_i \},\ {-}(1/2\mu)
      \|Ly\|^2 \right).
\]
In particular, because $h$ is a separable sum of Huber functions,
$W = [I\ {-}I]^T$, $w$ is the constant vector of all ones, and
$L = \sqrt{\eta} I.$ Since $\mu<0$, it follows that
\[
  \partial h^\pp(y, \mu)
  = \partial
  \left(
    \max \left\{
      \|y\|_\infty,  -(\eta/2\mu)\|y\|^2
    \right\}
  \right) (y, \mu).
\]
For the set
$\set{ v_1, \ldots, \, v_{2m+1} } := \left\{ y_1, \ldots, y_m , -y_1,
  \ldots, -y_m, - \frac{\eta}{2\mu }\|y\|^2 \right\},$ let
$J(y, \mu) := \set{ j | |v_j| = \max_{i=1,\ldots,2m+1}|v_i|} $ be the
set of maximizing indices. Then
\[
  \partial h^\pp(y, \mu) = \conv \set{ \nabla v_j | j \in J(y, \mu)},
\]
where $\conv$ denotes the convex hull operation. More concretely, precisely the following terms are contained in the convex hull above:
\begin{itemize}
\item
  $\left( {-}\frac{\eta}{\mu}y,\, \frac{\eta}{2\mu^2}\|y\|_2^2
  \right)$ if ${-}\frac{\eta}{2\mu }\|y\|^2 \ge \|y\|_\infty$;
\item $\left( \sign\, (y_i) \cdot e_i,\, 0 \right)$ if $i \in [m]$
  and $ |y_i| = \|y\|_\infty \ge -\frac{\eta}{2\mu }\|y\|^2
  $,
\end{itemize}
where $e_i$ is the $i$th standard basis vector. Note that if an
optimal solution to \eqref{eqn:target} exists, then \cref{thm:
  perrecovery} tells us that
$ ( {-}(\eta/\mu) y,\, (\eta/2\mu^2) \|y\|^2 ) $ must be included in
this convex hull, otherwise it is impossible to have
$(b-Ax, 1) \in \partial h^\pp(y, \mu).$

In summary, \cref{thm: perrecovery} tells us that to find an
optimal solution $x$ for \eqref{eqn:target}, we need to solve a linear
program to ensure that
$(b-Ax, 1) \in\conv \set{ \nabla v_j | j \in J(y, \mu) }$ subject
to the optimal support of $x$, as determined by~\eqref{ex: align}. In
cases where the size of the support is expected to be small (as might
be expected with a 1-norm objective), this required linear program can
be solved efficiently. 

\section{Numerical experiment: sparse robust regression} \label{sect:numerical}

To illustrate the usefulness of the primal-from-dual recovery
procedure implied by \cref{thm:lagrange}, we continue to examine the
sparse robust regression problem~\eqref{eq:sparse-huber}, considered
by Aravkin et al.~\cite{aravkin2013variational}. The aim is to find a
sparse signal (e.g., a spike train) from measurements contaminated by
outliers.
These experiments have been performed with the following data:
$m=120,$ $n=512,$ $\sigma = 0.2$, $\eta = 1$, and $A$ is a Gaussian
matrix. The true solution $\xbar\in \{-1,0,1\}$ is a spike train which
has been constructed to have 20 nonzero entries, and the true noise
$b-A\xbar$ has been constructed to have 5 outliers.

We compare two approaches for solving
problem~\eqref{eq:sparse-huber}. In both, we use Chambolle and Pock's
(CP) algorithm~\cite{cp}, which is primal-dual (in the sense of
Lagrange duality) and can be adapted to solve both the primal
problem~\eqref{eq:sparse-huber} and its perspective dual~\eqref{eq:
  fpolar_hub}.  Other numerical methods could certainly be applied to
either of these problems, such as Shefi and Teboulle's dual
moving-ball method~\cite{shefi2016rate}. We note that a primal-only
method, for example, applied to \eqref{eq: fpolar_hub}, would require
us to use the methods of \cref{sect: recovery_ex} rather than
\cref{thm:lagrange} for the recovery of a primal solution. %

The CP method applied to problem~\eqref{eq:9} at each iteration $k$ computes
\begin{align*}
  y^{k+1}&:=\prox_{\alpha_y f^\star}
              \big(y^k + \alpha_y A[2x^k - x^{k-1}] \big)
\\x^{k+1}&:=\prox_{\alpha_x g}(x^k - \alpha_x A\T y^{k+1} ),
\end{align*}
where
$\prox_{\alpha f}(x) := \argmin_y \{ f(y) + \frac{1}{2\alpha}\|
x-y\|_2^2\}$.  The positive scalars $\alpha_x$ and $\alpha_y$ are
chosen to satisfy $\alpha_x \alpha_y \|A\|^2 <1$.  Setting
$f=\delta_{h(b-\cdot) \le \sigma}$, and $g = \|\cdot \|_1$
yields the primal problem \eqref{eq:sparse-huber}. In this case, the
proximal operators $\prox_{\alpha f^\star}$ and $\prox_{\alpha g}$ can
be computed using the Moreau identity, i.e.,
\begin{align*}
\prox_{\alpha f^\star} (x)  &= x -  \prox_{(\alpha f^\star)^\star} (x) = x - \alpha \Pi_f ( x/\alpha) \\
\prox_{\alpha g} (y)  &= y - \prox_{(\alpha g)^\star} (y) = y - \Pi_{\alpha \mathbb{B}_\infty} (y/\alpha),
\end{align*}
where $\Pi_f$ is the projection onto the sublevel set in the
definition of $f$ and $\Pi_{\alpha \mathbb{B}_\infty}$ is the
projection onto the infinity-norm ball of radius $\alpha$. We
implement $\Pi_f$ using the \texttt{Convex.jl}~\cite{convex} and
\texttt{Gurobi} \cite{gurobi} software packages.

On the other hand, to apply CP to the perspective dual
problem~\eqref{eq: fpolar_hub}, one instead takes
$f=(\|\cdot\|)^\circ =\|\cdot\|_\infty$ and $g= \delta_\Qscr$, where
$\Qscr$ is the constraint set for \eqref{eq: fpolar_hub}, and take $A$
to be the corresponding adjoint to the operator
in~\eqref{eq:sparse-huber}.  To compute $\prox_{\alpha_y g}$, which is
the projection onto $\Qscr$, we solve the SOCP \eqref{eq:projsub}
using Gurobi. To evaluate $\prox_{\alpha f^\star}$, we again use
the Moreau identity and project onto level sets of $\|\cdot\|_1$.

\begin{figure}[t]
  \centering\small
  \begin{tabular}{@{}c@{}c@{}}
    \includegraphics[width=.49\textwidth]{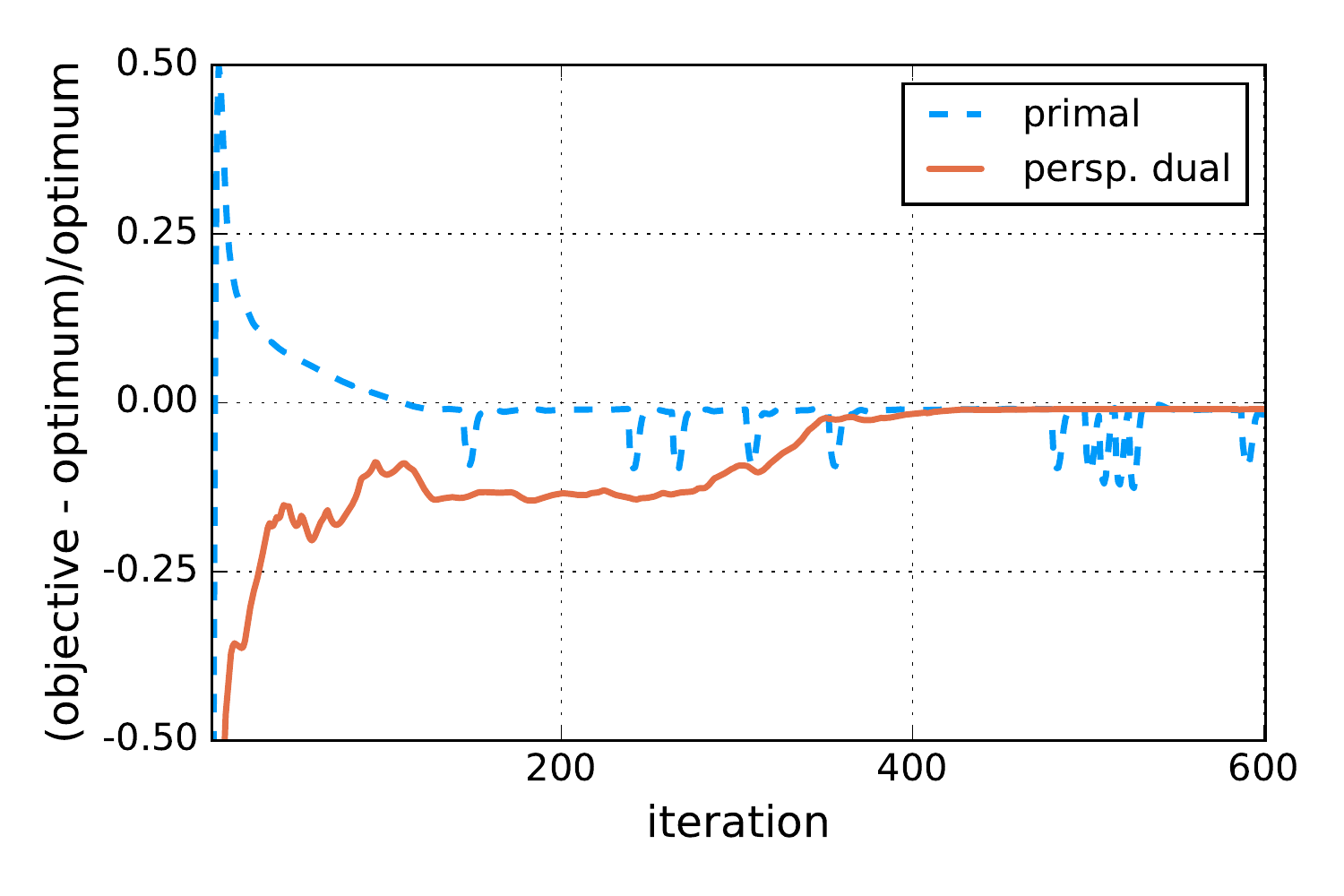}
   &\includegraphics[width=.49\textwidth]{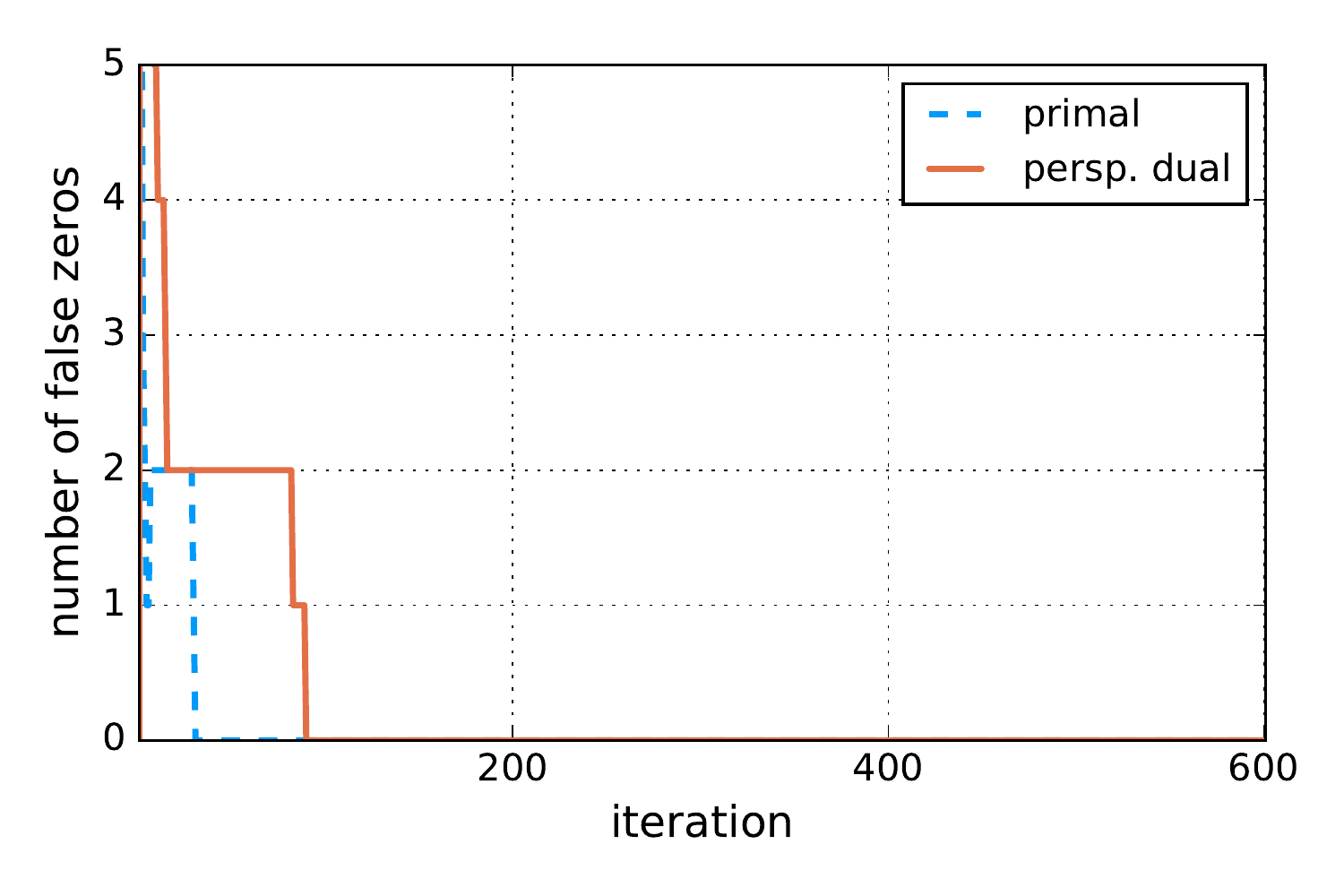}
  \\[-6pt] (a) Normalized objective values & (c) False zeros in iterates
  \\\includegraphics[width=.49\textwidth]{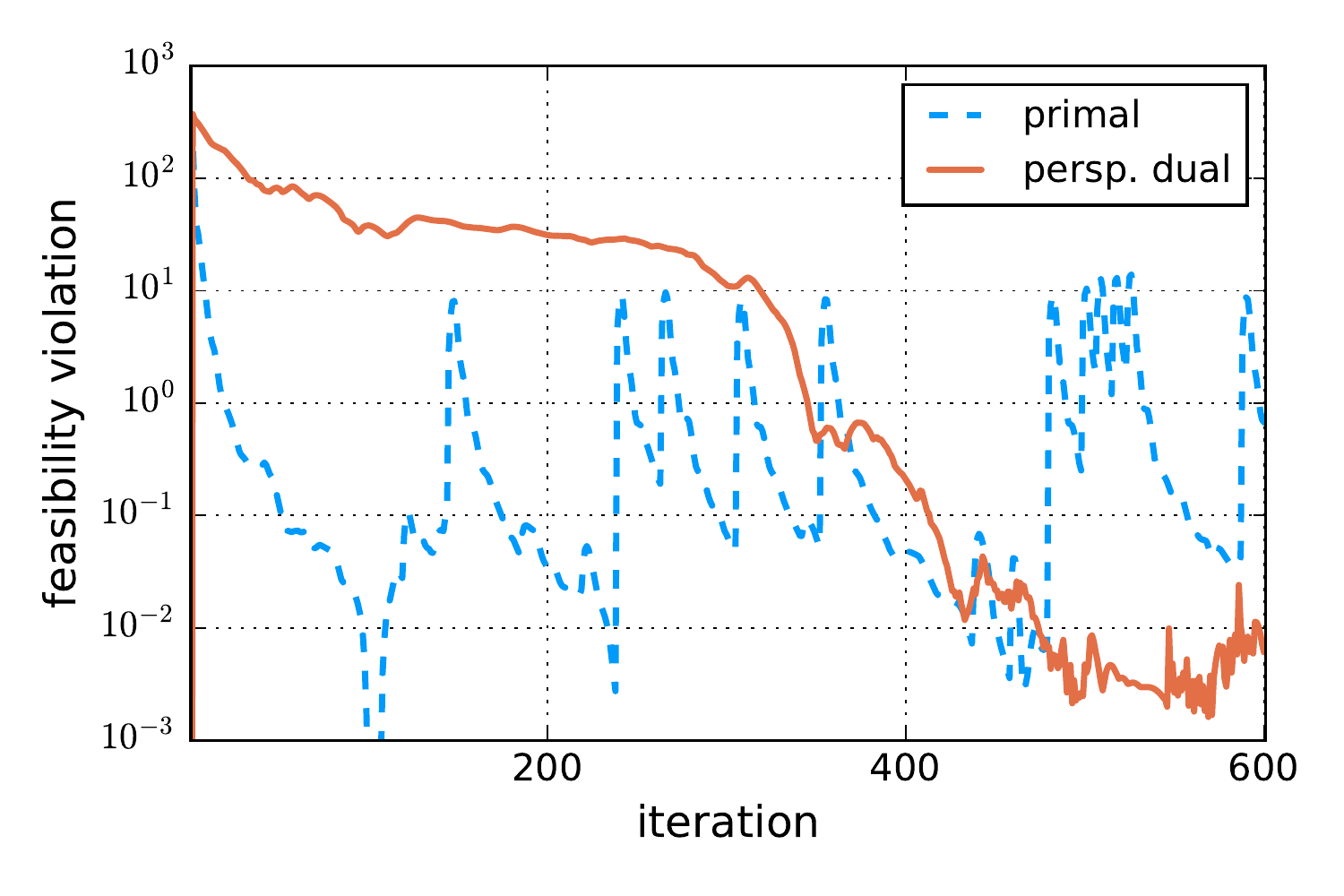}
   &\includegraphics[width=.49\textwidth]{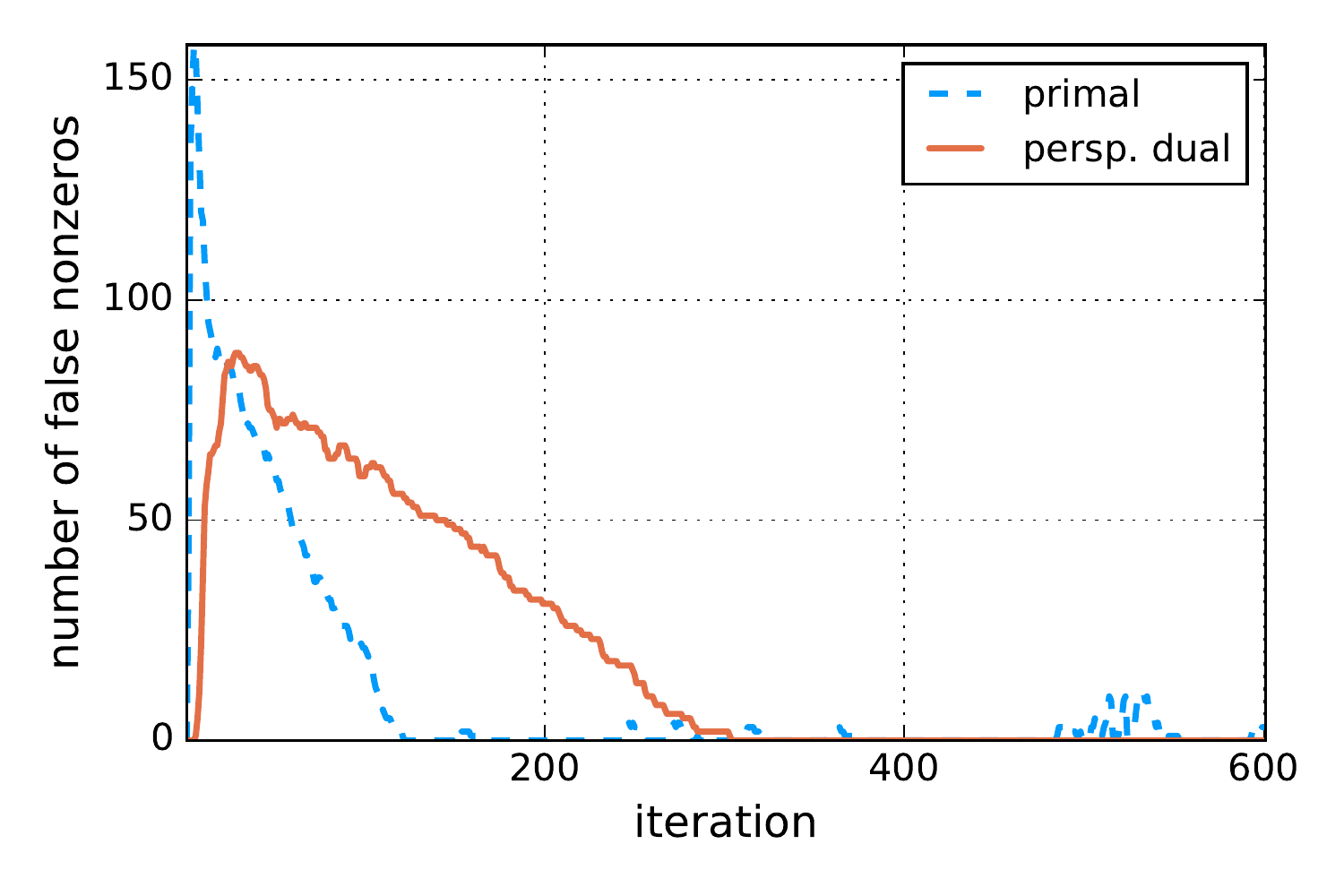}
  \\[-6pt] (b) Feasibility violations for iterates & (d) False nonzeros in iterates
  \end{tabular}
  \caption{\small The CP algorithm applied to sparse robust regression
    (\cref{sect:numerical}). Dashed lines indicate CP applied to the
    primal problem~\eqref{eq:sparse-huber}, and solid lines indicate
    CP applied to its perspective dual~\eqref{eq: fpolar_hub}
    where the primal solution is recovered via the method of
    \cref{thm:lagrange}. Plots show (a) normalized deviation of
    objective value $\|x^k\|_1$ from optimal value $\|\xbar\|_1$;
    (b) infeasibility measure $\max(h(b-Ax^k) - \sigma, 0 )$ for
    iterate
    $x^k$; %
    (c) number false zeros in iterate $x^k$ relative to $\xbar$;
    (d) number of false nonzeros in iterate $x^k$ relative to
    $\xbar$.  }
  \label{fig:CP}
\end{figure}

\cref{fig:CP} compares the outcomes of running CP on the primal and
perspective dual problems.  This experiment exhibited similar behavior
when run 500 times with different realizations of the random data, and
so here we report on a single problem instance.  Note that performing
an iteration of CP on the perspective dual is significantly faster
than performing an iteration of CP on the primal because $\Pi_\Qscr$
can be computed much more efficiently than $\Pi_f$ (see the discussion
in \cref{sect: plq}). This also appears to make convergence of CP on
the perspective dual more stable, as seen in
\cref{fig:CP}(a). \cref{fig:CP}(c)-(d) illustrate the sparsity
patterns of the iterates $x_{k}$ relative to those
$\xbar$. Notably, we recover the correct sparsity patterns using
\cref{thm:lagrange}. The recovery procedure outlined in \cref{ex:
  one_huber} also recovers the correct sparsity pattern, when applied
to the final perspective dual iterate.

\section{Discussion} \label{sect: discussion}

Gauge duality is fascinating in part because it shares many symmetric
properties with Lagrange duality, and yet Freund's 1987 development of
the concept flows from an entirely different principle based
on polarity of the sets that define the gauge functions.  On the other
hand, Lagrange duality proceeds from a perturbation argument, which
yields as one of its hallmarks a sensitivity interpretation of the
dual variables. The discussion in \cref{sect: gauge_sensitivity}
reveals that both duality notions can be derived from
the same Fenchel-Rockafellar perturbation framework. The derivation of
gauge duality using this framework appears to be its first application
to a perturbation that does not lead to Lagrange duality. This new link
between gauge duality and the perturbation framework establishes a
sensitivity interpretation for gauge dual variables, which has not been available until now.

One motivation for this work is to explore alternative formulations of optimization problems that might be computationally advantageous for certain problem classes. The phase-retrieval problem, based on an SDP formulation, was a first application of ideas from gauge duality for developing large-scale solvers \cite{FriedlanderMacedo:2016}. That approach, however, was limited in its flexibility because it required gauge functions. The discussions of \cref{sec:perspective-duality} pave the way to new extensions, such as different models of the measurement process, as described in \cref{sec:glms}.

Another implication of this work is that it establishes the foundation
for exploring a new breed of primal-dual algorithms based on
perspective duality. Our own application of Chambolle and Pock's
primal-dual algorithm \cite{cp} to the perspective-dual problem,
together with a procedure for extracting a primal estimate, is a first
exploratory step towards developing variations of such methods. Future
directions of research include the development of such algorithms,
along with their attendant convergence properties and an understanding
of the classes of problems for which they are practicable.

\section*{Acknowledgments}

We are grateful to Patrick Combettes for pointing us to recent
comprehensive work on properties of the perspective function and its
applications \cite{Combettes2017,Combettes2016}. Our sincere thanks to
two anonymous referees who provided an extensive list of corrections
and suggestions that helped us to arrive at several strengthened
results and to streamline our presentation.

\bibliographystyle{abbrv}
\bibliography{perdualbib}

\begin{thebibliography}{10}

\bibitem{aravkin2016level}
A.~Y. Aravkin, J.~V. Burke, D.~Drusvyatskiy, M.~P. Friedlander, and S.~Roy.
\newblock Level-set methods for convex optimization.
\newblock {\em arXiv:1602.01506}, 2016.

\bibitem{aravkin2013variational}
A.~Y. Aravkin, J.~V. Burke, and M.~P. Friedlander.
\newblock Variational properties of value functions.
\newblock {\em {SIAM} J. Optim.}, 23(3):1689--1717, 2013.

\bibitem{stablespline}
A.~Y. Aravkin, J.~V. Burke, and G.~Pillonetto.
\newblock Linear system identification using stable spline kernels and {PLQ}
  penalties.
\newblock In {\em 52nd {IEEE} Decis. Contr. P.}, pages 5168--5173, Dec 2013.

\bibitem{banerjee2005clustering}
A.~Banerjee, S.~Merugu, I.~S. Dhillon, and J.~Ghosh.
\newblock Clustering with {Bregman} divergences.
\newblock {\em J. Mach. Learn. Res.}, 6(Oct):1705--1749, 2005.

\bibitem{boydvan}
S.~Boyd and L.~Vandenberghe.
\newblock {\em Convex Optimization}.
\newblock Cambridge University Press, 2004.

\bibitem{robust_rec}
E.~J. Cand{\`e}s, J.~Romberg, and T.~Tao.
\newblock Robust uncertainty principles: exact signal reconstruction from
  highly incomplete frequency information.
\newblock {\em {IEEE} Trans. Inform. Theory}, 52(2):489--509, Feb 2006.

\bibitem{CRT}
E.~J. Cand{\`e}s, J.~K. Romberg, and T.~Tao.
\newblock Stable signal recovery from incomplete and inaccurate measurements.
\newblock {\em Comm. Pure Appl. Math.}, 59(8):1207--1223, 2006.

\bibitem{cp}
A.~Chambolle and T.~Pock.
\newblock A first-order primal-dual algorithm for convex problems with
  applications to imaging.
\newblock {\em J. Math. Imaging. Vis.}, 40(1):120--145, 2011.

\bibitem{Combettes2017}
P.~L. Combettes.
\newblock Perspective functions: Properties, constructions, and examples.
\newblock {\em Set-Valued Var. Anal.}, pages 1--18, 2017.

\bibitem{Combettes2016}
P.~L. Combettes and C.~L. M{\"u}ler.
\newblock Perspective functions: Proximal calculus and applications in
  high-dimensional statistics.
\newblock {\em J. Math. Anal. Appl.}, 2016.

\bibitem{don}
D.~Donoho.
\newblock Compressed sensing.
\newblock {\em IEEE Trans. Inform. Theory}, 52(4):1289--1306, 2006.

\bibitem{clarke_direct}
D.~Drusvyatskiy, A.~Ioffe, and A.~Lewis.
\newblock Clarke subgradients of directionally {L}ipschitzian stratifiable
  functions.
\newblock {\em Math. Oper. Res.}, 40(2):328--349, 2015.

\bibitem{freund}
R.~M. Freund.
\newblock Dual gauge programs, with applications to quadratic programming and
  the minimum-norm problem.
\newblock {\em Math. Program.}, 38(1):47--67, 1987.

\bibitem{FriedlanderMacedo:2016}
M.~P. Friedlander and I.~Mac\^edo.
\newblock Low-rank spectral optimization via gauge duality.
\newblock {\em {SIAM} J. Sci. Comput.}, 28(3):A1616--A1638, 2016.

\bibitem{gaugepaper}
M.~P. Friedlander, I.~Macedo, and T.~K. Pong.
\newblock Gauge optimization and duality.
\newblock {\em {SIAM} J. Optim.}, 24(4):1999--2022, 2014.

\bibitem{gurobi}
I.~Gurobi~Optimization.
\newblock Gurobi optimizer reference manual, 2015.

\bibitem{Hub}
P.~Huber.
\newblock {\em Robust Statistics}.
\newblock Wiley, 1981.

\bibitem{nelder1972generalized}
J.~A. Nelder and R.~J. Baker.
\newblock Generalized linear models.
\newblock {\em Encyclopedia of statistical sciences}, 1972.

\bibitem{ipmethods}
Y.~Nesterov and A.~Nemirovskii.
\newblock {\em Interior-point polynomial algorithms in convex programming},
  volume~13.
\newblock SIAM, 1994.

\bibitem{rockafellar}
R.~T. Rockafellar.
\newblock {\em Convex Analysis}.
\newblock Princeton University Press, 1972.

\bibitem{rockafellar1974}
R.~T. Rockafellar.
\newblock {\em Conjugate duality and optimization}.
\newblock SIAM, 1974.

\bibitem{Rock:88}
R.~T. Rockafellar.
\newblock First- and second-order epi-differentiability in nonlinear
  programming.
\newblock {\em Trans. Amer. Math. Soc.}, 307(1):75--108, May 1988.

\bibitem{rockafellarwets}
R.~T. Rockafellar and R.~J.-B. Wets.
\newblock {\em Variational Analysis}, volume 317.
\newblock Springer Science \& Business Media, 2009.

\bibitem{shefi2016dual}
R.~Shefi and M.~Teboulle.
\newblock A dual method for minimizing a nonsmooth objective over one smooth
  inequality constraint.
\newblock {\em Math. Program.}, 159(1-2):137--164, 2016.

\bibitem{shefi2016rate}
R.~Shefi and M.~Teboulle.
\newblock On the rate of convergence of the proximal alternating linearized
  minimization algorithm for convex problems.
\newblock {\em EURO J. Comp. Optim.}, 4(1):27--46, 2016.

\bibitem{tropp}
J.~Tropp.
\newblock Just relax: convex programming methods for identifying sparse signals
  in noise.
\newblock {\em IEEE Trans. Inform. Theory}, 52(3):1030--1051, 2006.

\bibitem{convex}
M.~Udell, K.~Mohan, D.~Zeng, J.~Hong, S.~Diamond, and S.~Boyd.
\newblock Convex optimization in {J}ulia.
\newblock {\em SC14 Workshop on High Performance Technical Computing in Dynamic
  Languages}, 2014.

\bibitem{elas_net}
H.~Zou and T.~Hastie.
\newblock Regularization and variable selection via the elastic net.
\newblock {\em J. R. Stat. Soc. Ser. B Stat. Methodol.}, 67(2):301--320, 2005.

\end{thebibliography}

\appendix

\section{Proof of \eqref{eq:unit and zero}} \label{sec:sigma0_facts}
We prove each fact in succession.
\begin{enumerate}
\item ($\Uscr_{\kappa}^\circ=\Uscr_{\kappa^\circ}$). By definition of the polar gauge and the polar cone, we have $y \in \Uscr_{\kappa^\circ}$ if and only if
  \[\sup\set{\ip x y|\kappa(x)\le 1} \le 1 \iff y \in \Uscr_\kappa^\circ.\]
  
\item ($\Uscr_{\kappa}^\infty=\Hscr_{\kappa}$). Suppose
  $x \in \Hscr_\kappa$. Then for any $u \in \Uscr_\kappa$ and
  $\lambda >0$, by sublinearity of $\kappa$ we have
  $\kappa(u+\lambda x) \le \kappa (u) + \lambda \kappa(x) \le 1 +
  \lambda \cdot 0=1.$ Thus $x \in \Uscr_{\kappa}^\infty$, and
  $\Hscr_\kappa \subseteq \Uscr_\kappa^\infty$. Suppose now that
  $y \in \Uscr_{\kappa}^\infty\setminus \Hscr_\kappa$. Then in
  particular, $\kappa \left( y/\kappa(y) + \lambda y \right) \le 1$
  for all $\lambda >0$. But then by positive homogeneity,
  $\left( 1/\kappa(y) + \lambda \right) \kappa(y) \le 1$, for all
  $\lambda >0$. This is a contradiction since $\kappa(y) >0$, so we
  conclude that $\Hscr_\kappa =\Uscr_\kappa^\infty$.
  
  \item $((\dom \kappa)^\circ=\Hscr_{\kappa^\circ}).$ By positive homogeneity of $\kappa$ and the definition of the polar gauge, $y \in \Hscr_{\kappa^\circ}$ if and only if
  \[  \sup_{\kappa(x) \le 1} \ip x y = 0 \iff \sup_{\kappa(x) <\infty } \ip {x} y = 0
  \iff y \in (\dom \kappa)^\circ.  \]
  
  \item ($\Hscr_{\kappa}^\circ=\cl \dom \kappa^\circ$). Apply the third equality, replacing $\kappa$ by $\kappa^\circ$, and then take polars on both sides. This concludes the proof.
\end{enumerate}
\qed

\section{Proof of \cref{prop:strict_feas}} \label{sec:proof-crefpr}
With no loss in generality, we can assume that $\sigma>0$, because if
$\sigma =0$, we use the convention~\eqref{eq:replacement} and its
implication~\eqref{eq:convention-sigma-zero}.

First suppose that the primal \eqref{eq:gauge-primal} is relatively
strictly feasible.  A point $u$ lies in the domain of $p$ if and only
if the system
\[
  (u,0,0)\in M
\begin{bmatrix}
  w \\\lambda
\end{bmatrix}+(\epi \rho)\times \Uscr_{\kappa},
\textt{where}
M:= \begin{bsmallmatrix}
  A & -b\\
  0 & -\sigma\\
  -I & 0
\end{bsmallmatrix},
\]
is solvable for $(w,\lambda)$. Thus the set $(\dom p)\times \{0\}\times \{0\}$ coincides with 
\begin{equation}\label{eqn:inter_main}
  L\bigcap\left(\range M+(\epi \rho)\times \Uscr_{\kappa}\right),
\end{equation}
where $L:=\set{(a,b,c) | b=0,\,c=0}$ is a linear subspace. We aim to
show $(0,0,0)$ is in the relative interior of \eqref{eqn:inter_main},
which will show $0 \in \ri\dom p$. Use \cite[Lemma
7.3]{rockafellar} and \cite[Theorem 7.6]{rockafellar} to obtain
\begin{align*}
  \ri \epi \rho&=\set{(z,r)\in (\ri\dom \rho)\times \R| \rho(z)<r} \\
  \ri \Uscr_{\kappa}&=\set{x\in \ri\dom \kappa | \kappa(x)<1}.
\end{align*}
From relative strict feasibility of \eqref{eq:gauge-primal}, the fact that $\sigma>0$, and again \cite[Theorem 7.6]{rockafellar}, we deduce
existence of an $x\in \ri\dom \kappa$ with
$b-Ax\in \ri\dom \rho$ and $ \rho(b-Ax)<\sigma.$
Fix a constant $r>\kappa(x)$ and define the pair
$(w,\lambda):=(x/r,1/r)$. Then we immediately have
$(b\lambda -Aw,\sigma\lambda)\in \ri\epi\rho$ and $\kappa(w)<1$.
It follows that the vector
$-M\begin{bsmallmatrix}w\\\lambda\end{bsmallmatrix}$ lies in
$(\ri\epi\rho)\times \ri \Uscr_{\kappa}$. Thus $(0,0,0)$ lies in the
intersection
\begin{equation}\label{eqn:inter_ri}
 L\bigcap\left(\range M+ \left[ (\ri\epi \rho)\times \ri\Uscr_{\kappa} \right] \right).
\end{equation}	
Use \cite[Theorem 6.5, Corollary 6.6.2]{rockafellar} to deduce that
\eqref{eqn:inter_ri} is the relative interior of the
intersection \eqref{eqn:inter_main}. Thus $y=0$ lies in the relative
interior of $\dom p$ as claimed.

Next, suppose that the gauge dual  \eqref{eq:gauge-dual} is strictly feasible.
By definition of $F^\star,$ the tuple $(w ,\lambda)$ lies in the domain of $v_d$ if and only if 
\[
  (w,0, -\lambda) \in \left( \dom \kappa^\circ \times \epi(\sigma \rho^\circ - \ip{b}{\cdot} + 1 ) \right) 
  - \range B,
  \textt{with}
  B:= \begin{bsmallmatrix} A^T \\ I \\ 0 \end{bsmallmatrix}.
\]
Thus $\dom v_d$ is linearly isomorphic to the intersection
\begin{equation}\label{eqn:inter_dual}
  L' \bigcap \left( 
    \left( \dom \kappa^\circ \times \epi(\sigma \rho^\circ - \ip{b}{\cdot} + 1 ) \right) 
    - \range B \right)  ,
\end{equation}
where $L'$ is the linear subspace $L':=\{(a,b,c)\mid b=0\}$.
However, by \cite[Lemma 7.3]{rockafellar}, relative strict feasibility of the dual \eqref{eq:gauge-dual}  amounts to the inclusion
\[
  (0, 0, 0 ) \in  \left( \ri\dom\kappa^\circ \times \ri\epi(\sigma \rho^\circ - \ip{b}{\cdot} + 1 ) \right) -\range B.
\]
Strict feasibility of~\eqref{eq:gauge-dual} implies, via \cite[Corollary 6.5.1, Corollary 6.6.2]{rockafellar}, that  $(0,0,0)$ is in the relative interior of the intersection \eqref{eqn:inter_dual}, and thus 
$ 0 \in\ri\dom v_d$, as claimed.

Finally, the exact same arguments, but with relative interiors replaced by interiors, will prove the claims relating strict feasibility and interiority. This concludes the proof.

\end{document}